\documentclass[10pt, a4paper]{amsart}
 
\usepackage{dynkin-diagrams} 
\newcommand{\dynpair}[2]{%
  $#1$\;
  \raisebox{-0.45\height}{#2}%
}
\usepackage{adjustbox}
\usepackage[T1]{fontenc}
\usepackage{microtype}

\usepackage{amsmath}								
\usepackage{amssymb}
\usepackage{amsthm}
\usepackage{thmtools}
\usepackage{amscd}
\usepackage{amsfonts}
\usepackage{dsfont}
\usepackage{bbm}
\usepackage{stmaryrd}

\usepackage{extarrows}

\usepackage{color}
\usepackage{xcolor}
\usepackage{comment}
\usepackage{mathrsfs}
\usepackage{bm}
\usepackage{tcolorbox}
\definecolor{GoetheBlue}{RGB}{0,97,143}

\usepackage[bookmarks=true, colorlinks=true, linkcolor=blue!50!black,
citecolor=GoetheBlue, urlcolor=GoetheBlue, pdfencoding=unicode]{hyperref}

\usepackage{tikz}									
\usetikzlibrary{matrix}
\usetikzlibrary{patterns}
\usetikzlibrary{matrix}
\usetikzlibrary{positioning}
\usetikzlibrary{decorations.pathmorphing}
\usetikzlibrary{cd}
\usetikzlibrary{trees}
\usepackage{hyperref}
\usepackage{graphicx}
\usepackage{caption}
\usepackage{subcaption}
\usepackage{cleveref}

\usepackage{fullpage}
\usepackage{enumitem}
\usepackage{mathtools}
\usepackage{parskip}
\usepackage[new]{old-arrows}
\usepackage[left=3.5cm,right=3.5cm,top=4cm,bottom=4cm,footskip=1cm
]{geometry}
\usepackage{multicol}
\usepackage[inkscape=false]{svg}

\makeatletter
\def\thm@space@setup{%
  \thm@preskip=5pt \thm@postskip=5pt
}
\makeatother

\newtheorem*{theorem*}{Theorem}

\newtheorem{theorem}{Theorem}[section]
\newtheorem{lemma}[theorem]{Lemma}
\newtheorem{proposition}[theorem]{Proposition}
\newtheorem{corollary}[theorem]{Corollary} 
\newtheorem{conjecture}[theorem]{Conjecture}

\theoremstyle{definition}
\newtheorem{definition}[theorem]{Definition}
\newtheorem{example}[theorem]{Example}

\newtheorem{remark}[theorem]{Remark}

\newtheoremstyle{myitemstyle}						
	{}			
	{}			
	{}			
	{}			
	{}			
	{.}			
	{ }			
	{}			
\theoremstyle{myitemstyle}
\newtheorem{myitemthm}{}

\newcommand{\B}{{\mathds B}}
\newcommand{\R}{\mathds{R}}
\newcommand{\T}{\mathds{T}}
\newcommand{\Z}{\mathds{Z}}
\newcommand{\N}{\mathds{N}}

\newcommand{\C}{\mathds{C}}

\renewcommand{\P}{\mathds{P}}
\newcommand{\calM}{\mathcal{M}}

\newcommand{\even}{{\mathrm{even}}}
\newcommand{\odd}{{\mathrm{odd}}}
\newcommand{\refinedDressian}{\Dr'(B_3,P_3)}

\newcommand{\strongExchangeFan}{\Sigma^{S}_{B_3}}

\DeclareMathOperator{\supp}{supp}

\newcommand{\trop}{{\operatorname{trop}}}
\DeclareMathOperator{\Dr}{Dr}

\DeclareMathOperator{\Gr}{Gr}

\newcommand{\conv}{\mathrm{conv}}
\newcommand{\stab}{\mathrm{stab}}
\newcommand{\mc}[1]{{\mathcal{#1}}}
\newcommand{\Poly}[2]{\Delta(W(#1)/#2)}

\title{Minuscule Coxeter Dressians}
\author{Andreas Gross}
\address{Institut f\"ur Mathematik, Goethe--Universit\"at Frankfurt,
60325 Frankfurt am Main, Germany}
\email{gross@math.uni-frankfurt.de}

\author{Kevin K\"uhn}
\address{Technische Universit\"at Berlin,
10587 Berlin, Germany}
\email{kuehn@math.tu-berlin.de}

\author{Dante Luber}
\address{Universit\"at Paderborn,
33098 Paderborn, Germany}
\email{dluber@math.uni-paderborn.de}
\date{}

\begin{document}

\begin{abstract}
    We study tropical prevarieties associated with minuscule representations of reductive groups. Recently, the defining equations of these prevarieties were used to detect the strong exchange property for Coxeter matroids. In type $A$, this reduces to the well known symmetric exchange property for usual matroids, and the tropical prevarieties in question are the classical Dressians. We study the analogous objects in all Lie types beyond type $A$ which admit minuscule parabolics. Our main results show that points on these prevarieties induce subdivisions of Coxeter matroid polytopes into cells which posses the strong exchange property. Furthermore, we establish conditions for a subdivision to be strong matroidal which are detectable on the induced subdivision of lower dimensional skeleta. We include results concerning explicit computations of the relevant polyhedral complexes and their $f$-vectors.  
\end{abstract}

\maketitle

\section*{Introduction}
\renewcommand{\thetheorem}{\Alph{theorem}}

Among the most well studied algebraic varieties are Grassmannians. The Grassmannian $\Gr(k,n)$ parametrizes $k$-dimension linear subspaces of $\C^n$, and admits a natural embedding
\[
    p\colon \Gr(k,n)\hookrightarrow \P\C^{\binom{n}{k}}
\]
into projective space known as the \emph{Plücker embedding}. The ideal which vanishes on the image of the Pl\"ucker embedding is generated by the \emph{Grassmann-Plücker relations}, given by
\[
    \sum_{j\in J\setminus I} \pm x_{I\cup\{j\}}\cdot x_{J\setminus\{j\}}
\]
for subsets $I\in \binom{[n]}{k-1}$ and $J\in\binom{[n]}{k+1}$. The signs can be expressed explicitly as well, but will not be relevant in this work. 

For $[U]\in\Gr(k,n)$, one can define a natural subset $\mc M\subseteq \binom{[n]}{k}$ consisting of all $I$ such that the coordinate $x_I$ does not vanish on $p([U])$. The set $\mc M$ forms the bases of a rank $k$ \emph{matroid} on the groundset $[n]$. In fact, one of the main motivations behind matroid theory is to study sets $\mc M$ arising in this manner. 

However, not all matroids come from conditions which can be traced back to a vanishing pattern of Pl\"ucker coordinates. More generally, matroids may be defined in terms of set systems which are closed under \emph{symmetric exchange}. Concretely, a collection $\mathcal{M}\subset\binom{[n]}{k}$ forms the bases of a matroid if and only if for all $A,B\in\mathcal{M}$ and $a\in A\setminus B$, there exists $b\in B\setminus A$ such that $(A\setminus \{a\})\cup \{b\}$ and $(B\setminus \{b\})\cup \{a\}$ are mutually elements of $\mathcal{M}$. This is a purely set theoretic abstraction of the exchange property shared by bases of finite dimensional vector spaces and spanning trees of graphs. The foundational works on  matroid theory aimed to elucidate this phenomena    \cite{whitney1931non,whitney2}, and develop axiomatics for more general notions of independence \cite{nakasawa_book}. However, most matroids do not capture independence in any kind of linear-algebraic or graphical setting \cite{Nelson_2018}. 

Whether a set system $\mathcal{M}\subset \binom{[n]}{k}$ defines the bases of a matroid can also be detected via polyhedral geometry. That is, the convex hull
\begin{equation}
\label{eq:intro Grassmann-Pluecker}
    \Delta(\mc M)=\conv\{e_I\colon I\in \mc M\}\subset \mathbb{R}^n,
\end{equation}

(where $e_I=\sum_{i\in I} e_i$ is a sum of standard basis vectors) is a \emph{matroid polytope} if its edges are in the direction $e_i-e_j$ with $i\neq j$. In other words, $\mathcal{M}$ is a matroid if and only if the edges of $\Delta(\mathcal{M})$ are elements of the root system $A_{n-1}$. The set-theoretic and polyhedral definitions are amongst numerous cryptomorphic characterizations for matroids. For a complete treatment of matroid theory, some standard references are  \cite{Oxley2011,White1986}, as well as the survey paper \cite{Oxley2003}.

If $[U]$ is defined over a valued field $K$, for example the field of fractions $\C(t)$, then one obtains a vector $\mu\in \R^{\binom{[n]}{k}}$ by applying the valuation of $K$ to the Plücker coordinates $x_I$ of $p([U])$ (in general, some coordinates $\mu_I$ can be infinite, but for this exposition we assume they are all finite). The vector $\mu$ can be viewed as a height function $\mu$ on the vertices of the hypersimplex $\Delta(\binom{[n]}{k})=\Delta_{n,k}$, which is the polytope associated with the \emph{uniform matroid}. Then, one is interested in the $\mu$-induced \emph{regular subdivision} of $\Delta_{n,k}$. 

Speyer \cite{speyer_tropical_linear_spaces}, showed that the following are equivalent:
\begin{enumerate}[label=(\alph*)]
    \item the regular subdivision of $\Delta_{n,k}$ induced by $\mu$ is matroidal. That is, all its cells are matroid polytopes, 
    \item $\mu$ satisfies the $3$-term (tropical) Grassmann--Plücker relations. That is, for every $(k-4)$-element subset $I$ of $[n]$ and all choices of four distinct indices $a,b,c,d\in [n]\setminus A$, the minimum
    \[
        \min\{
            \mu_{I\cup\{{a,b\}}}+\mu_{I\cup\{c,d\}},
            \mu_{I\cup\{{a,c\}}}+\mu_{I\cup\{b,d\}},
            \mu_{I\cup\{{a,d\}}}+\mu_{I\cup\{b,c\}}
            \}
    \]
    is attained at least twice.
\end{enumerate}
Note that the tropical Grassmann--Plücker relations are the tropicalizations of the Grassmann--Plücker relations \eqref{eq:intro Grassmann-Pluecker} obtained by replacing sums by minima and products by sums, respectively. 
Rincón \cite{rincon_isotropical} showed, based on work of Murota \cite{Murota}, that these two equivalent conditions are in fact equivalent to $\mu$ satisfying all (tropical) Grassmann-Plücker relations. That is, that for all choices of subsets $I,J\subseteq [n]$ with $|I|=k-1$ and $|J|=k+1$ the minimum
\[
    \min_{j\in J\setminus I} (\mu_{I\cup\{j\}}+ \mu_{J\setminus\{j\}})
\]
is attained at least twice. Expressed differently, $\mu$ satisfies the $3$-term Plücker relations if and only if it is a valuated matroid in the sense of Dress and Wenzel \cite{DressWenzel_valuatedmatroids}.

Grassmannians are special cases of type $A$ flag varieties, which parameterize sequences of nested linear subspaces in a common ambient vector space. Similarly, the matroids discussed above are special cases of \emph{flag matroids}, which combinatorially abstract the aforementioned concepts from linear algebra. Flag matroids are synonymous with type $A$-\emph{Coxeter matroids}. Coxeter matroids, defined in detail below, generalize usual matroids to any Lie type and admit a polyhedral-geometric definition.

\emph{Generalized} flag varieties are of the form $G/P_G$ for a semisimple algebraic group $G$ and a parabolic subgroup $P_G$ of $G$, and exist for Lie types beyond type $A$.

The aim of this paper is to generalize the results of Speyer and Murota-Rincón to generalized flag varieties. The most urgent questions in this direction are the following:
\begin{itemize}
    \item What are the analogues of the Grassmann-Plücker relations?
    \item What are the analogues of the $3$-term Grassmann-Plücker relations.
    \item What are the analogues of valuated matroids in the setting of more general Coxeter matroids?
\end{itemize}

In type $D_n$, this has been answered by Wenzel \cite{WenzelDelta} and Rincón \cite{rincon_isotropical}: Grassmann-Plücker relations are replaced by Wick relations and the $3$-term Grassmann--Plücker relations by the $4$-term Wick relations. 

To define Grassmann--Plücker relations more generally, one needs coordinates on the flag variety $G/P_G$. This means one needs an embedding $G/P_G\hookrightarrow \P^n$ into a projective space. There are plenty such embeddings, but the most natural ones arise from representations of $G$. We are thus looking for representations of $G$ on a vector space $V$ that come with canonical coordinates. There is a natural splitting $V\cong \oplus_\chi V_\chi$ of $V$ into weight spaces $V_\chi$, which are permuted by the \emph{Weyl group} $W$ of $G$. So if all weight spaces $V_\chi$ are $1$-dimensional, and the action of $W$ on the appearing weights $\chi$ is transitive, there are canonical coordinates on $V$ (at least up to scaling), and hence on $G/P_G$. By definition, this happens if and only if the representation is \emph{minuscule}.

Given a minuscule representation of $G$ on $V$, the coordinates on $V$ are indexed by the cosets in $W/P$ for the parabolic subgroup $P$ of $W$ corresponding to $P_G$. We thus obtain a subvariety of the projective space $\mathbb P\C^{W/P}$. Recently, Calvert, Dermenjian, Fink, and Smith \cite{Fink_et_al} gave explicit equations for $G/P_G$ in this embedding. Calvert et.\ al.\ tropicalize these relations to obtain the \emph{strong exchange equations}, equations which capture the strong exchange property for Coxeter matroids of Borovik, Gelfand, and White \cite{BorovikGelfandWhite:2003}. We denote the set of strong exchange equations by $\mc F_{\Phi,P}$, where $\Phi$ is the root system associated to $G$. The strong exchange equations allow us to define the Coxeter Dressian 
\[
    \Dr(\Phi, P) = V^\trop(\mc F_{\Phi,P}),
\]
which consists of all vectors in $\T^{W/P}$ ($\T=\R\cup\{\infty\}$, the tropical semifield) satisfying all strong exchange relations. That is, $\Dr(\Phi, P)$ is the tropical prevariety obtained by intersecting the tropical hypersurfaces we get from each element of $\mc F_{\Phi,P}$. 

In type $A$ and $D$, the strong exchange equations recover the tropical Grassmann--Plücker and Wick relations, respectively. 
The weights $\chi$ from above span a polytope $\Delta(W/P)$, whose vertices are indexed by $W/P$. This polytope plays the role of the hypersimplex $\Delta_{n,k}$ in the work to follow. The strong exchange relations are supported on the faces $\Delta(W/P)$. Of particular importance to us are the strong exchange relations supported on faces of diameter $2$, which we call the \emph{distance-$2$ relations}. In type $A$ and $D$, these recover the $3$-term Grassmann--Plücker relations and $4$-term Wick relations, respectively.

Given ``tropical coordinates'' $\mu\colon W/P\to \R$, we obtain a regular subdivision of $\Delta(W/P)$, and can detect whether it is (Coxeter) matroidal, that is whether its cells are Coxeter matroid polytopes. We can now state our first main result, generalizing Speyer's theorem in type A.

\begin{theorem}
[Theorems \ref{thm:subidivions type A}, 
\ref{thm:subidivions type Dn},
\ref{thm:subidivions type DP1},
\ref{thm:subidivions type E6}, and
 \ref{thm:subidivions type E7}, and \Cref{ex:strong_fails_fink}]
\label{intro_thm:subdivisions}
Let $(\Phi,P)$ be a minuscule type not equal to $(B_n,P_n)$ and let $\mu\in \R^{W(\Phi)/P}$. The following are equivalent:
\begin{enumerate}[label=(\alph*)]
    \item $\mu$ satisfies the distance-$2$ strong exchange relations,
    \item the subdivision of $\Delta(W(\Phi)/P)$ induced by $\mu$ is matroidal.
\end{enumerate}
In type $(B_n,P_n)$, this fails even for $n=3$.
\end{theorem}

Type $(B_n,P_n)$ is special among the minuscule types, as there are vertices of $\Delta(W(B_n)/P_n)$ that are not adjacent, but still differ by a root (strictly speaking, this also happens in type $(C_n,P_1)$, but we will see in Section \ref{sec:type_C} that this case is trivial). However, it turns out that the theories of valuated Coxeter matroids of types $(B_n,P_n)$ and $(D_{n+1},P_{n+1})$ are completely equivalent. 

\begin{theorem}[\Cref{prop:B_n_vs_D_n+1}]
There is a linear bijection between the Coxeter Dressians $\Dr(B_n,P_n)$ and $\Dr(D_{n+1},P_{n+1})$.
\end{theorem}

The set of heights $\mu$ on $\Delta(W(B_n)/P_n)$ that induce a matroidal subdivision is strictly larger than the Coxeter Dressian $\Dr(B_n,P_n)$, as we define it. It has been shown in \cite{Valuated_Delta_Subdivisions}, that this strictly larger set is defined by tropical equations of degree $4$ that are supported on the $4$-dimensional faces.

What is remarkable about our proof of \Cref{intro_thm:subdivisions} is that it uses a criterion for when a subdivision of \emph{any} polytope does not introduce new edges. 

\begin{theorem}[\Cref{thm:1-skeletal subdivision}]
Let $\mathscr{S}$ be a subdivision of a polytope $\Delta$. Then the edges of $\mathscr{S}$ coincide with the edges of $\Delta$ if and only if this is true for the restriction of $\mathscr{S}$ to all faces of $\Delta$ whose edge graph has diameter $2$.
\end{theorem}

We can generalize the Murota--Rincón result as well, in a mostly positive way.

\begin{theorem}[
    Theorems \ref{thm:subidivions type A},
    \ref{thm:distance-2 Dn}, and
    \ref{thm:b_4_equations}]
Let $(\Phi,P)$ be a minuscule type not equal  to $(E_7,P_7)$, and let $\mu\in \R^{W(\Phi)/P}$. The following are equivalent:
\begin{enumerate}[label=(\alph*)]
    \item $\mu$ satisfies the distance-$2$ strong exchange relations
    \item $\mu$ satisfies all strong exchange relations.
\end{enumerate}
\end{theorem}
In type $(E_7,P_7)$ we conjecture the equivalence to hold as well.

\subsection*{Further directions}
Dressians of type $A$ are deeply connected to numerous areas of mathematics. Many of these have natural analogs in arbitrary Lie type, in the minuscule and non-minuscule case. This leads to the following questions:
\begin{itemize}
    \item \emph{Tropical linear spaces} generalize tropicalizations of linear spaces defined over fields and are dual to matroidal subdivisions of hypersimplices. In this sense, Dressians of type $A$ do not only parametrize matroidal subdivisions, but also tropical linear spaces. Do analogs of tropical linear spaces exist outside type $A$? If so, how do these structures relate to more general Coxeter matroid subdivisions? Existing results in type $D$ \cite{rincon_isotropical} suggest that this is subtle.
    \item Can the theory of valuated minuscule Coxeter matroids be meaningfully generalized to arbitrary tracts (see Remark \ref{rem:baker-bowler stuff})? For example, there should be a notion of \emph{oriented} (valuated) minuscule Coxeter matroids.
    \item What happens when we drop the assumption that we are working with a minuscule parabolic?  For type $A$, questions of this nature have been addressed in \cite{BEZ,JLLO}. In type $C$, non-minuscule maximal parabolics have been studied in the context of tropical symplectic Grassmannians and Dressians \cite{balla_olarte_symplectic}.
\end{itemize}

\subsection*{Software and computations}\label{sec:computations} We utilize computational tools at verious points in this work. Our results be reproduced (and further experimented with) using \texttt{jupyter} notebooks found at the following Github repository:
\begin{center}
    \href{https://github.com/danteluber/strong_coxeter_dressians}{\url{https://github.com/danteluber/strong_coxeter_dressians}}
\end{center}
\subsection*{Acknowledgements} 
We would like to thank Kieran Calvert, Aram Dermenjian, Alex Fink and Ben Smith for insightful discussions about their work on strong exchange equations. Furthermore, we thank Michael Joswig, Lars Kastner, Igor Makhlin, and Martin Ulirsch for helpful input at various stages of the project. 

Andreas Gross has received funding from the Deutsche Forschungsgemeinschaft (DFG, German Research Foundation) TRR 326 \emph{Geometry and Arithmetic of Uniformized Structures}, project number 444845124 and from the Marie-Sk\l{}odowska-Curie-Stipendium Hessen (as part of the HESSEN HORIZON initiative).
Kevin K\"uhn received support by Deutsche Forschungsgemeinschaft (DFG) through "Symbolic Tools in Mathematics and their Application" (TRR 195, project ID 286237555). 
Dante Luber is supported by the DFG Sachbeihilfe "Rethinking tropical linear algebra: Buildings, bimatroids, and applications", (SPP 2458, project ID 539867663), within Combinatorial Synergies, and also received support from the Engineering and Physical Sciences Research Council (grant number EP/X001229/1).

\section{Background on Coxeter Matroids}
\renewcommand{\thetheorem}{\thesection.\arabic{theorem}}
\renewcommand{\theHtheorem}{\thesection.\arabic{theorem}}

\label{sec:root systems anf Weyl groups} 
In this section we give the background on Coxeter matroids from a group theoretic and polyhedral perspective. We will begin with the relevant notions concerning reflection groups and root systems. The standard reference for Coxeter matroids in full generality is \cite{BorovikGelfandWhite:2003}. 

\subsection{Root Systems and Weyl Groups}

\emph{Root systems} are special collections of vectors corresponding to mirrors over hyperplanes in $\R^n$ with deep connections to representation theory. More formally, a root system is a finite collection  of \emph{roots} $\Phi$, that is of vectors in an inner product space $(V, \langle - , - \rangle)$ such that 
\begin{itemize}
    \item $\Phi$ spans $V$ and $0\notin \Phi$
    \item $\R\alpha\cap \Phi=\{\alpha,-\alpha\}$ for all $\alpha\in \Phi$,
    \item for all $\alpha\in \Phi$, the reflection $\tau_\alpha$ at the hyperplane orthogonal to $\alpha$ satisfies $\tau_\alpha(\Phi)=\Phi$,
    \item for all $\alpha,\beta\in\Phi$, we have
    \[
        \langle \alpha,\check\beta\rangle\in \Z,
    \]
    where $\check\beta=\frac{2\beta}{\langle \beta,\beta\rangle}$ is the \emph{coroot} associated to $\beta$.
\end{itemize}
The \emph{Weyl group} $W$ of $\Phi$ is the finite reflection group generated by the reflections $\tau_\alpha$ for $\alpha\in \Phi$.

After choosing a generic linear functional, $l\colon V\to \R$, one can define the set of \emph{positive} roots
\[
    \Phi^+=\{\alpha\in \Phi : l(\alpha)>0 \} 
\]
and the set of \emph{simple} roots $\Phi^s$, which are those positive roots that cannot be expressed non-trivially as sums of positive roots. Both $\Phi^+$ and $\Phi^s$ depend on the choice of $l$, but any two choices are equivalent up to the action on $W$. 

A root system is \emph{irreducible}, if it cannot be written as a nontrivial union $\Phi=\Phi_1\cup \Phi_2$ with $\langle \alpha,\beta\rangle=0$ for all $\alpha\in \Phi_1$ and $\beta\in \Phi_2$.

A \emph{weight} is a vector $\omega\in V$ with 
\( 
    \langle \omega,\check\alpha\rangle\in \Z
\)
for all $\alpha\in\Phi$. A weight $\omega$ is \emph{dominant}, if $\langle \omega,\alpha\rangle\geq 0$ for all $\alpha\in \Phi^s$. 
The Weyl group $W$ acts on weights and every weight has a unique  dominant weight in its orbit. 
Let $\Phi^s=\{\alpha_1,\ldots, \alpha_n\}$. For every $1\leq i \leq n$ there is a \emph{fundamental weight} $\omega_i$ with 
\(
    \langle \omega_i, \check\alpha_j\rangle=\delta_{ij},
\)
where $\delta_{ij}$ is the Kronecker delta.

A weight $\omega$ is \emph{minuscule}, if 
\(
    |\langle \omega, \check\alpha\rangle|\leq 1
\)
for all $\alpha\in \Phi$. The trivial weight $\omega=0$ is clearly minuscule. All nonzero dominant minuscule weights need to be fundamental weights. The stabilizers $P_i$ of the fundamental weights $\omega_i$ under the $W$-action are the \emph{maximal standard parabolic subgroups} of $W$. They are in bijective correspondence to the simple roots, the root $\alpha_i$ corresponding to $P_i=\langle \tau_{\alpha_1},\ldots,\widehat{\tau_{\alpha_i}},\ldots, \tau_{\alpha_n}\rangle$.

Irreducible root systems admit a classification by their \emph{Lie type}. For this work, the relevant Lie types will be $A_n,B_n, D_n,C_n, E_6,E_7$, because only in those types do minuscule weights exist.

\begin{example}\label{ex:weights classical}
We recall the standard description of the Lie classical types and their minuscule weights.
\begin{description}[leftmargin=1em]
\item[$\mathbf{A_{n-1}}$] 
In type $A_{n-1}$, the roots are given by the differences $e_i-e_j$ of distinct standard basis vectors $e_i$ in $\R^n$. The space $V$ is the span of these roots, that is the subspace of $\R^n$, where the coordinates sum to $0$. The reflection at $e_i-e_j$ swaps the $i$-th and $j$-th coordinate and thus the Weyl group generated by these reflections is the symmetric group $S_{n}$ acting on $\R^{n}$ via permuting the variables. One choice of simple roots is $\alpha_i=e_{i+1}-e_{i}$ for $1\leq i\leq n-1$, whose associated reflections correspond to the neighborhood transpositions $\tau_1=(1~2),\ldots, \tau_{n-1}=(n-1~ n)$. 

Consider the dualizing isomorphism
\[
V\mapsto V^*=\R^n/\R\mathbf 1, v\mapsto \langle v,-\rangle,
\]
where $\mathbf 1$ is the all-$1$ vector. Under the dualizing isomorphism, the fundamental weight $\omega_i$ is mapped to the class of $e_{[i]}=\sum^{i}_{j=1}e_j$. All fundamental weights are minuscule.

\item[$\mathbf {B_n}$]
In type $B_n$, the roots are the differences of standard basis vectors $e_i-e_j\in \R^n$ for $i\neq j$, as in type $A_{n-1}$, together with the vectors $\pm e_i \pm e_j$ for $i\neq j$ and $\pm e_i$ for $1\leq i\neq n$. The reflection at $e_i+e_j$ interchanges the $i$-th and $j$-th coordinate and then flips their signs. The reflection at $e_i$ flips the sign at the $i$-th coordinate. The Weyl group $W$ in type $B_n$ acts on $\{\pm e_i\}$, and labeling $e_i$ by $i$ and $-e_i$ by $-i$, we can identify $W$ as the \emph{signed permutation group} (or \emph{hyperoctahedral group}), that is as subgroup of the symmetric group on $[n]\cup -[n]$ consisting of all permutations satisfying $\sigma(-k)=-\sigma(k)$ for all $k\in [n]$.

The standard choice of simple roots in type $B_n$ is given by $\alpha_i=e_i-e_{i+1}$ for $1\leq i\leq n-1$ and $\alpha_n=e_n$. The only minuscule fundamental weight is $\omega_n= (\frac 12,\ldots, \frac 12)$.

\item[$\mathbf{C_n}$]
The root system $C_n$ is \emph{dual} to $B_n$, that is it is obtained by interchanging the roles of roots and coroots. So the roots are given by $e_i-e_j$, $\pm e_i\pm e_j$ for $i\neq j$, and $2e_i$ for $1\leq i\leq n$. The standard choice of simple roots for $C_n$ is $\alpha_i=e_i-e_{i+1}$ for $1\leq i\leq n-1$ and $\alpha_n=2e_n$. As coroots are scalar multiples of roots, this does not change the Weyl group, which is the signed permutation group as in type $B_n$. The minuscule weights, however, are different from type $B_n$. In type $C_n$, only the fundamental weight $\omega_1=e_1$ is minuscule.

\item[$\mathbf{D_n}$] In type $D_n$, the roots are given by those of type $A_{n-1}$, that is $e_i-e_j\in \R^n$ for $i\neq j$, together with the roots $\pm e_i\pm e_j$ for $i\neq j$, that also appeared in type $B_n$ and $C_n$. Since the reflection at $\pm e_i\pm e_j$ is an exchange of coordinate plus two sign flips, the Weyl group in type $D_n$ is the subgroup of the signed permutation group consisting of all signed permutations that flip an even number of signs. The standard choice of simple roots in type $D_n$ is $\alpha_i=e_i-e_{i+1}$ for $i\leq n-1$, and $\alpha_{n}=e_{n-1}+e_n$. 

There are three minuscule fundamental weights, namely
\[\omega_1=e_1\text{, }  \omega_{n-1}=\left(\frac 12,\ldots,\frac 12,-\frac 12\right) \text{, and }
    \omega_n=\left(\frac 12,\ldots, \frac 12\right).\]
\end{description}
\end{example}

\begin{example}
For each Lie type, we have an associated \emph{Dynkin diagram}. The diagrams associated with the relevant Lie types given in Figure \ref{fig:minuscule_dynkin}. Each vertex of the Dynkin diagram corresponds to a simple root. We mark in black precisely those vertices, whose associated fundamental weights are minuscule. We draw an edge between vertices $\alpha_i,\alpha_j$, if the order of $\tau_{\alpha_i}\tau_{\alpha_j}$ in $W$ is $3$. Finally, we draw an arrow from $\alpha_i$ to $\alpha_j$, if the order of $\tau_{\alpha_i}\tau_{\alpha_j}$ in $W$ is $4$, and $\langle \alpha_i ,\alpha_i\rangle > \langle \alpha_j,\alpha_j \rangle$.

\begin{figure}[htbp]
    \centering
      \begin{tabular}{cc}

\dynpair{A_{n}:}{
\begin{tikzpicture}
  \dynkin[labels={\alpha_1,\alpha_{2},\alpha_{n-1},\alpha_{n}},
          label directions={,,,,,,},
          scale=2.5] A{**...**}
\end{tikzpicture}}
&
\dynpair{B_{n}:}{
\begin{tikzpicture}
  \dynkin[labels={\alpha_1,\alpha_{2},\alpha_{n-1},\alpha_{n}},
          label directions={,,,,,,},
          scale=2.5] B{oo...o*}
\end{tikzpicture}}
\\[1.2em]
\dynpair{C_{n}:}{
\begin{tikzpicture}
  \dynkin[labels={\alpha_1,\alpha_{2},\alpha_{n-1},\alpha_{n}},
          label directions={,,,,,,},
          scale=2.5] C{*o...oo}
\end{tikzpicture}}
&
\dynpair{D_{n}:}{
\begin{tikzpicture}
  \dynkin[labels={\alpha_1,\alpha_2,\alpha_{n-3},\alpha_{n-2},\alpha_{n},\alpha_{n-1}},
          label directions={,,,right,,},
          scale=2.5] D{*o...oo**}
\end{tikzpicture}}
\\[1.2em]
\dynpair{E_{6}:}{
\begin{tikzpicture}
  \dynkin[labels={\alpha_1,\alpha_2,\alpha_{3},\alpha_{4},\alpha_{5},\alpha_{6}},
          label directions={,,,,,},
          scale=2.5] E{*oooo*}
\end{tikzpicture}}
&
\dynpair{E_{7}:}{
\begin{tikzpicture}
  \dynkin[labels={\alpha_1,\alpha_2,\alpha_{3},\alpha_{4},\alpha_{5},\alpha_{6},\alpha_{7}},
          label directions={,,,,,,},
          scale=2.5] E{oooooo*}
\end{tikzpicture}}
\end{tabular}

    \caption{Dynkin diagrams for minuscule Lie types}
    \label{fig:minuscule_dynkin}
\end{figure}
\end{example}

\subsection{Coxeter Matroids in Minuscule Lie Types and Their Polytopes}
In this section we introduce Coxeter matroids and their polyhedral incarnations. Let $\Phi$ be a semisimple root system with Weyl group $W$, acting on an inner product space $V$ via reflection. Make a choice of simple roots $\Phi^s\subseteq \Phi$ and let $\alpha\in \Phi^s$. Let $\omega$ be the fundamental weight associated to $\alpha$ and let $P=\stab(\omega)$ be the associated standard parabolic group. We obtain the \emph{orbit polytope} as the convex hull
\[
    \Delta(W/P)=\conv(\{w\cdot \omega:w\in W\}).
\]
It is a lattice polytope in the lattice of all weights of $\Phi$. 
As $\omega$ is stabilized by $P$, the vertices of $\Delta(W/P)$ are in bijection with cosets in $W/P$. 

A nonempty collection of cosets $\mathcal{M}\subset W/P$ is a \emph{Coxeter matroid} of type $(\Phi,P)$  if all edges of the polytope 
\[
    \Delta(\mathcal{M})= \conv(\left\{w\cdot \omega :w\in\mathcal{M}\right\})
\]
are parallel to elements of $\Phi$.

\begin{example}
\label{ex:matroids classical}
We continue with \Cref{ex:weights classical} and give explicit descriptions of Coxeter matroids in the minuscule classical types.
\begin{description}[leftmargin=1em]
\item[$\mathbf A_{n-1}$] 
Consider the weight $\omega_k=e_{[k]}\in \R^n/\R\mathbf 1$ in type $A_{n-1}$, with notation as in \Cref{ex:weights classical}. The stabilizer of $\omega_k$ is $P_k=S_k\times S_{n-k}$. The polytope $\Delta(W(A_{n-1})/P_k)$ equals the convex hull
\[
    \conv\{\sigma\cdot e_{[k]} : \sigma\in S_n\}= \conv\left\{e_I: I\subseteq \binom{[n]}{k}\right\},
\]
where $e_I=\sum_{i\in I}e_i$. In other words, $\Delta(W(A_{n-1})/P_k)$ is the hypersimplex $\Delta_{n,k}$. 

As $P_k$ is also the stabilizer of $[k]$ with respect to the transitive $S_n$-action on $\binom{[n]}{k}$, we may identify $W(A_{n-1})/P_k$ with $\binom{[n]}{k}$. Therefore, a Coxeter matroid of type $(A_{n-1},P_k)$ is a collection $\mc M$ of subsets of $\binom{[n]}{k}$ such that the polytope
\[
    \conv\{e_I: I\in \mc M\}
\]
has edges parallel to $e_i-e_j$ for $i\neq j$. By the classifications of matroids in terms of matroid polytopes \cite[Theorem 4.1]{gelfand_serganova}, we conclude that a Coxeter matroid $\mc M$ of type $(A_{n-1},P_k)$ is the same as a matroid of rank $k$ on the ground set $[n]$, and that $\Delta(\mc M)$ is the associated matroid polytope.

The case where $\Phi=A_3$ and $P_2=\langle \tau_1,\tau_3\rangle$ is illustrated in Figure \ref{fig:delta24}, where the vertices are labeled by the coset representative of shortest length in $W(A_3)$, written as a product of adjacent transpositions $\tau_i=(i,i+1)$, which are the reflections corresponding to the simple roots $\alpha_i$. The embedding is given by the orbit of $\omega=(1,1,0,0)$. Edges correspond to pairs of vertices that differ in exactly two coordinates.\begin{figure}[ht]
    \centering
    \includegraphics[width=0.45\linewidth]{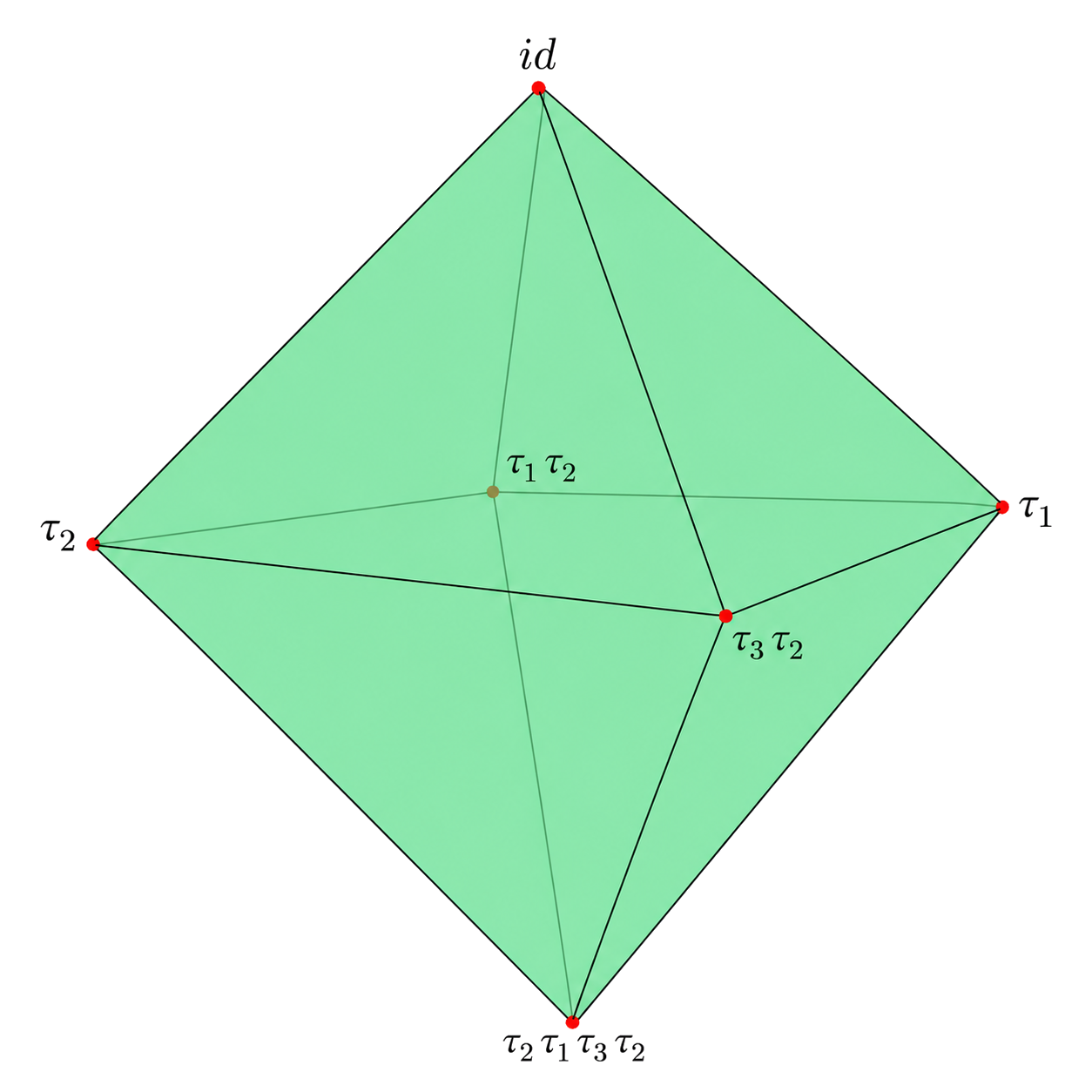}
    \caption{The Coxeter matroid polytope of $\Delta(W(A_3)/P_2)=\Delta_{4,2}$.}
    \label{fig:delta24}
\end{figure}
\item[$\mathbf{B_n}$]
The orbit of the unique minuscule fundamental weight $\omega_n=\left(\frac 12,\ldots,\frac 12\right)$ is the set
\(
    \{\left(\pm\frac 12,\ldots,\pm\frac 12\right)\}
\)
of vertices of the $n$-cube $\Delta(W(B_n)/P_n)=\left[-\frac 12,\frac 12\right]^n$. Only recording which coordinates are negative, we obtain a bijection between the vertices of $\Delta(W(B_n)/P_n)$ and the power set $2^{[n]}$ of $[n]$. By the  classification of Delta-matroids by their polytopes, it follows that a Coxeter matroid $\mc M$ of type $(B_n,P_n)$ is the same as a Delta-matroid (also called Lagrangian or symmetric matroid) on the ground set $[n]$.

\item[$\mathbf{C_n}$]
The orbit of the unique minuscule fundamental weight $\omega_1=e_1$ is $\{\pm e_i: 1\leq i\leq n\}$. The convex hull $\Delta(W/P_1)$ of the orbit is an $n$-dimensional \emph{cross polytope}
\[
    \Diamond_n=\conv\{\pm e_i: 1\leq i\leq n\}.
\]
The line segments between any two vertices is parallel to a root in $C_n$, so the condition on $\mc M\subseteq W(C_n)/P_1$ to be a Coxeter matroid is vacuously satisfied. Therefore, Coxeter matroids of type $(C_n,P_1)$ can be identified with nonempty subsets of $[n]\cup -[n]$

\item[$\mathbf{D_n}$] 
In type $D_n$, there are two cases. We first treat the case of the minuscule fundamental weight $\omega_1=e_1$. The orbit of this is $\{\pm e_i:1\leq i\leq n\}$, so we have $\Delta(W(D_n)/P_1)=\Diamond_n$ as in the case $(C_n,P_1)$. However, the edge directions $\pm e_i$ between antipodal vertices of the cross polytope are not parallel to roots in the $D_n$ case. Therefore, there is are nontrivial conditions for a subset $\mc M\subseteq W(D_n)/P_1$ to be a Coxeter matroid. As we will see in \Cref{prop:matroidal=strongly matroidal for Dn-P1}, $\mc M$ is a Coxeter matroid of type $(D_n,P_1)$ if and only if $\mc M$ either contains no antipodal pair of vertices of $\Delta(W(D_n)/P_1)$, or at least $2$ such pairs.

The second case is that of the fundamental weights $\omega_{n-1}$ and $\omega_n$. The polytope $\Delta(W/P_n)$ (resp\  $\Delta(W/P_{n-1})$) is the \emph{$n$-demicube} spanned by all vertices of the $n$-cube $\left[-\frac 12,\frac 12\right]$ with an even (resp.\ odd) number of negative coordinates. Labeling the vertices by elements in $2^{[n]}$ as before, we see that $W/P_n$ (resp.\ $W/P_{n-1}$) can be identified with all subsets of $[n]$ of even (resp.\ odd) cardinality. Coxeter matroids of type $(D_n,P_{n-1})$ and $(D_n,P_n)$ correspond precisely to the even Delta-matroids \cite[Theorem 2]{OrthogonalMatroids}.
\end{description}
\end{example}

\section{Strong Exchange and the Coxeter Dressian}\label{sec:strong_exchange_and _cox_dress}
In this section we recall the strong exchange property for Coxeter matroids, originally introduced by Borovik, Gelfand, and White \cite{BorovikGelfandWhite:2003}. We will explain how this property can be checked using tropical equations, the strong exchange equations, by work of Calvert, Dermenjian, Fink, and Smith \cite{Fink_et_al}. Finally, we use the strong exchange equations to define Coxeter Dressians. Throughout this section we fix an irreducible minuscule type $(\Phi,P)$ with Weyl group $W=W(\Phi)$.

\subsection{The Strong Exchange Property}
Consider any collection of cosets $\mathcal{M}\subseteq W/P$. Then $\mathcal{M}$ is a \emph{strong} Coxeter matroid, or $\mathcal{M}$ satisfies the \emph{strong exchange property}, if for all $A,B\in\mathcal{M}$, there exists a reflection $s\in\Phi$ such that the mirror defined by $s$ separates $A$ and $B$, and $sA,sB\in\mathcal{M}$. Observe that this definition mimics the symmetric exchange property for usual matroids. Any collection of cosets with the strong exchange property is a Coxeter matroid \cite[Theorem 6.1.2]{BorovikGelfandWhite:2003}, but the converse fails \cite[4.2.4]{BorovikGelfandWhite:2003}. 

We can restate the strong exchange property in terms of the polytope $\Delta(\mathcal{M})$: given two vertices $v,w$ of $\Delta(\calM)$, then there exist vertices $v',w'$ of $\Delta(\calM)$ such that 
\begin{itemize}
\item $v,v'$ and $w,w'$ respectively are adjacent,
\item the edge directions $v-v'$ and $w'-w$ are parallel.
\end{itemize}

\subsection{Strong Exchange Equations}
Let $\T=\R\cup\{\infty\}$ denote the tropical semifield with operations
\[
    x\oplus y=\min\{x,y\} \quad\text{and}\quad x\odot y=x+y.
\]
Let $\B=\{0,\infty\}$ be the Boolean semifield, a subsemifield of $\T$.

For a tropical polynomial $f=\bigoplus_{\alpha} c_\alpha\odot \mathbf x^\alpha\in \T[\mathbf x]$, where $\mathbf x=(x_1,\ldots, x_n)$ and we use multi-index notations, we say that $\mu\in \T^n\setminus\{(\infty,\dots,\infty)\}$ satisfies the tropical relation given by $f$ if the minimum 
\[
    \bigoplus_{\alpha} c_\alpha\odot \mu^{\alpha}=\min_\alpha (c_\alpha+\langle \alpha, \mu\rangle)
\]
is attained at least twice. We denote by $V^{\trop}(f)$ the set of $\mu\in \T^n\setminus\{(\infty,\dots,\infty)\}$ satisfying $f$. 
If $\mc M\subseteq [n]$ is a subset, we identify $\mc M$ with a vector in $\B^n$ which is zero at the coordinates index by $\mc M$ and $\infty$ otherwise. Therefore, it makes sense to say that $\mc M$ satisfies the tropical polynomial $f$.

Calvert, Dermenjian, Fink, and Smith \cite{Fink_et_al} introduced a collection of tropical equations $\mathcal{F}_{\Phi,P}\subset \B[x_A\mid A\in W/P]$, known as the \emph{(tropical) strong exchange equations}. A collection $\mc M\subseteq W/P$ satisfies all $f\in\mathcal{F}_{\Phi,P}$ if and only if $\mathcal{M}$ is a strong Coxeter matroid. These equations are square-free, quadratic, and  reflect the polyhedral structure of $\Delta(W/P)$, where the (degree-2)-monomials correspond to pairs of antipodes on a face. In particular, they are homogeneous in the grading where $x_A$ has as weight the corresponding vertex of $\Delta(W/P)$. We refrain from stating the strong exchange equations at this point; they are given in Equations 
\ref{eq:D_n_strong},
\ref{eq:B_n_strong},
\ref{eq:D_1_strong},
\ref{eq:E_6_strong}, and
\ref{eq:E_7_strong} in the subsections of the minuscule types. 

The tropicalization of a polynomial $f=\sum_\alpha c_\alpha \mathbf x^\alpha\in\C[\mathbf x]$ is the tropical polynomial $f^\trop=\bigoplus_{\alpha:c_\alpha\neq 0} \mathbf x^{\odot\alpha}$. We briefly outline how the set $\mc F_{\Phi,P}$ arises via this process. 
There exists a semisimple group $G$ corresponding to $\Phi$ and a parabolic $P_G\leq G$ corresponding to $P$. The group $G$ is really only well-defined up to isogeny, but the flag variety $G/P_G$ does not depend on any choices. By the classifications of irreducible representations of $G$ by highest weight \cite[Section 31.3]{humphreys}, the dominant minuscule weight $\omega$ corresponding to $P$ yields a representation of $G$ on a vector space $V$, together with a one-dimensional subspace $L$ whose stabilizer is $P_G$. Because the representation is minuscule, we have $V=\bigoplus_{A\in W/P} A\cdot L$, and thus obtain a well-defined $G$-equivariant embedding $G/P_G\hookrightarrow \P\C^{W(\Phi)/P}$ sending the coset $1\cdot P$ to $L$.  According to \cite[Theorem 1.1]{Fink_et_al}, there exists a finite set $\mc F^{\mathrm{alg}}_{\Phi,P}\subseteq \C[x_A:A\in W/P]$ of quadratic polynomials generating the vanishing ideal of $G/P_G$ such that the set $\mc F_{\Phi,P}$ from above is given by $\{f^\trop: f\in \mc F^{\mathrm{alg}}_{\Phi,P}\}$.

\begin{remark}
\label{rem:baker-bowler stuff}
These equations are similar in spirit to those studied in \cite{bakerbowler2016matroids,bakerbowlerpartial}, however these works only cover usual matroids. They consider different hyperfields, or \emph{tracts}, other than $\T$. Several cryptomorphic descriptions of minuscule $D_n$-matroids over tracts have been established in \cite{tong_donggyu_tracts}. We expect similar results to be possible in all minuscule types.
\end{remark}
\subsection{The Coxeter Dressian}
In this section we construct the relevant tropical-geometric objects. They are polyhedral complexes known as \emph{tropical prevarieties}, which are obtained by intersecting finitely many tropical hypersurfaces. Our references for tropical geometry are \cite{MaclaganSturmfels} and \cite{Joswig_ETC}.
\begin{definition}\label{def:coxeter_Dressian}
Let $(\Phi,P)$ be a minuscule type. Let $\mathcal{F}_{\Phi,P}$ be the collection of associated tropical strong exchange equations. The \emph{(strong) Coxeter Dressian} of $W/P$ is the tropical prevariety 
\[
    \Dr(\Phi,P)=V^\trop(\mc F_{\Phi,P})=\bigcap_{f\in\mathcal{F}_{\Phi,P}}V^{\trop}(f) \subseteq \T^{W/P}.
\]
We call points $\mu \in \Dr(\Phi,P)$ \emph{valuated Coxeter matroids of type $(\Phi,P)$}. In practice, we will mostly consider valuated Coxeter matroids with only finite entries.
\end{definition}

\begin{definition}
\label{def:affinely invariant}
Let $V$ be a real vector space and let $E\subset V$ be a finite set. We say that a collection 
$\mathcal{F}\subseteq \T[x_e| e\in E]$ of tropical polynomials with trivial coefficients is \emph{affinely invariant} if for any function $\mu: E\to \T$,
and for for all affine linear functions $\varphi: V \to \R$, we have
\[
    \mu \in V^{\trop}(\mathcal{F}) \text{ if and only if } \mu+\varphi|_E \in V^{\trop}(\mathcal{F})\, .
\]
\end{definition}

The principal source of examples for affinely invariant sets of polynomials are homogeneous ones in the sense of the following lemma, as follows directly from the definition.

\begin{lemma}
\label{lem:homogeneous implies affinely invariant}
Let $E$ be finite subset of a vector space $V$ and let $\mc F\subseteq \T[x_e:e\in E]$ be a set of tropical polynomials. Assume that the polynomials in $\mc F$ are homogeneous in the grading where $x_e$ has weight $e$. Then $\mc F$ is \emph{affinely invariant}.
\end{lemma}

\begin{theorem}
\label{thm:cells_strong_exchange}
Let $E\subset V$ be finite, and $\mu\colon E \to \T$, with induced subdivision $\mathscr{S}(\mu)$ of $\conv(\supp(\mu))$. Let $\mathcal{F}\subseteq \T[x_e| e\in E]$ be an affinely invariant set of tropical polynomials with trivial coefficients. If $\mu\in V^{\trop}(\mathcal{F}) \subseteq \T^E$, then for any cell $C \in \mathscr{S}(\mu)$, the function
\begin{align*}
    \mu_C\colon E \to \T, \quad e  \mapsto 
    \begin{cases}
        0 \ \ \ \text{ if }e \in C \\
        \infty \, \text{ else}
\end{cases}
\end{align*}
satisfies all tropical equations in $\mathcal{F}$. In other words, $\mu_C \in V^{\trop}(\mathcal{F)}$.
\end{theorem}

\begin{proof}
Consider a cell $C$ of $\mathscr{S}(\mu)$. Let $\tilde{\mu}= \mu+\varphi$ for some affine linear function $\varphi$ such that $\tilde{\mu}|_C=0$ and $\tilde{\mu}|_{S\setminus C} >0$. 
Since $\mathcal{F}$ is affinely invariant, $\tilde{\mu}\in V^{\trop}(\mathcal{F})$. Furthermore, since $\mu,\tilde{\mu}$ only differ by an affine linear function, their induced subdivisions agree.
Let $f \in \mathcal{F}$, then $f = \bigoplus_{\alpha}x^{\odot\alpha}$ for finitely many $\alpha \in \N^E$. Since this minimum is attained at least twice, when plugging in $\tilde{\mu}$, there are two cases:
\begin{itemize}
    \item If the evaluation $f(\tilde{\mu})=0$, then there are at least two monomials of $f$ supported on $C$. In that case, surely $f(\mu_C)$ = 0, and the minimum is attained at these monomials. Thus, $\mu_C \in V^{\trop}(f)$.
    
    \item If $f(\tilde{\mu})>0$, then all monomials of $f$ contain a variable in $E\setminus C$. It follows, that $f(\mu_C)= \infty$, and hence trivially $\mu_C$ tropically satisfies $f$.
\end{itemize}
\end{proof}

\section{One-skeletal subdivisions of Polytopes}\label{sec:one_skeletal}

In \Cref{thm:subidivions type A}, a global property of a subdivision of the hypersimplex, namely being matroidal, can be checked on the $3$-skeleton of the hypersimplex. This section generalizes this to subdivisions of arbitrary polytopes.

\begin{definition}
Let $\Delta$ be a polytope. A \emph{subdivision} $\mathscr{S}$ of $\Delta$ is a polyhedral complex with support $\Delta$, whose vertices are the vertices of $\Delta$.
We say that $\mathscr{S}$ is \emph{$1$-skeletal} if all of its edges are edges of $\Delta$.
\end{definition}

If $F$ is a face of a polytope $\Delta$ and $\mathscr{S}$ is a subdivision of $\Delta$, then we denote by $\mathscr{S}|_F$ the induced subdivision $\{\delta\in \mathscr{S}\mid \delta\subseteq F\}$ of $F$.

\begin{definition}
A face $F$ of a polytope $\Delta$ is a \emph{distance-$2$-witness} if there exists a vertex $v$ and two edges $\varepsilon_1$ and $\varepsilon_2$ of $F$ adjacent to $v$ whose other endpoints are not connected by an edge and such that $F$ is the minimal face of $\Delta$ containing both $\varepsilon_1$ and $\varepsilon_2$.
\end{definition}

\begin{remark}
A distance-$2$-witness $F$ witnesses that the vertices of $\varepsilon_1$ and $\varepsilon_2$ distinct from $v$ have distance $2$ in the edge graph of $\Delta$.
\end{remark}

\pagebreak
\begin{example}
\
\begin{enumerate}[label=(\alph*)]
   \item If $\Delta=\Delta_{n,k}=\Delta(W(A_{n-1})/P_k)$ is the hypersimplex, then its distance-$2$-witnesses are isomorphic to the hypersimplex $\Delta_{4,2}$. This is an octahedron and isomorphic to the cross polytope $\Diamond_3$. See Figure \ref{fig:delta24}.
    \item If $\Delta=\Delta(W(D_n)/P_n)$ is the $n$-demicube, then its distance-$2$-witnesses are isomorphic to the $4$-demicube, which in turn is isomorphic to the cross polytope $\Diamond_4$. The \emph{Schlegel diagram} of $\Diamond_{4}$ is displayed in Figure \ref{fig:D_4_schlegel}.
\end{enumerate}
\begin{figure}[ht]
    \centering
    \includegraphics[width=0.3\linewidth]{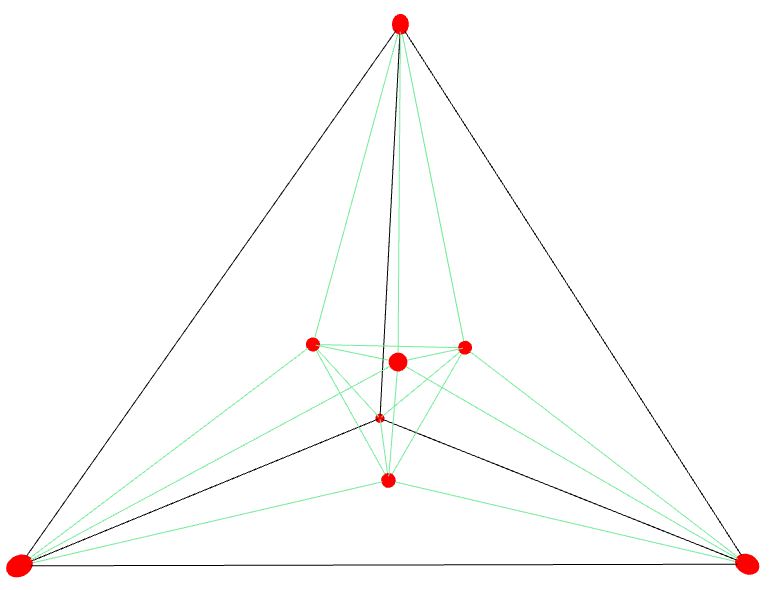}
    \caption{The Schlegel diagram of the cross polytope $\Delta(W(D_4)/P_4)$.}
    \label{fig:D_4_schlegel}
\end{figure}
\end{example}
\begin{theorem}
\label{thm:1-skeletal subdivision}
Let $\Delta$ be a polytope and let $\mathscr{S}$ be a subdivision of $\Delta$. Then $\mathscr{S}$ is $1$-skeletal if and only if $\mathscr{S}|_F$ is $1$-skeletal for all distance-$2$-witnesses $F$ of $\Delta$.
\end{theorem}
\begin{proof}
If $\mathscr{S}$ is $1$-skeletal, then so is $\mathscr{S}|_F$ for all faces $F$ of $\Delta$. Conversely, assume that $\mathscr{S}|_F$ is $1$-skeletal for all distance-$2$-witnesses $F$ of $\Delta$. Inducting on $\dim \Delta$, we may assume that $\mathscr{S}|_F$ is $1$-skeletal for all proper faces $F$ of $\Delta$. We may also assume that $\Delta$ is full-dimensional in $\R^n$. Now assume for contradiction that there exists an edge of $\mathscr{S}$ that is not an edge of $\Delta$. Let $\varepsilon$ be such an edge and $H$ a closed half-space with $\varepsilon\subseteq \partial H$, and let $(\varepsilon,H)$ be  chosen such that the number $|H\cap V(\Delta)|$  of vertices of $\Delta$ lying in $H$ is minimal. Because $\varepsilon$ intersects the relative interior of $\Delta$, there exists a $2$-dimensional face $\delta\in \mathscr{S}$ with $\varepsilon\subseteq \delta\subseteq H$.

Assume for contradiction that $\delta$ has an edge $\varepsilon'$ distinct from $\varepsilon$ that is not an edge of $\Delta$. Let $l_1$ be an affine linear function with $\{l_1\geq 0\}= H$ and let $w$ be a vertex of $\varepsilon'$ on which $l_1|_{\varepsilon'}$ is minimal (it might happen that $l_1$ is constant on $\varepsilon'$). Let $l_2$ be an affine linear function that attains its minimum on $\Delta$ precisely at $w$, and let 
\[
l\colon \R^d\to \R^2,\;\; x\mapsto (l_1(x),l_2(x)) \ .
\]
After slightly perturbing $l_2$, we may assume that $l$ is injective on $\delta$. Then $l(\delta)$ is a $2$-dimensional convex polygon and, by construction, $l(w)$ is a vertex of $l(\Delta)$, see \Cref{fig:subdivision proof}.

\begin{figure}[htbp]
    \centering
    \begin{tikzpicture}[
    dot/.style ={
        circle,
        black,
        fill=black,
        inner sep=.5mm
        },
    font=\footnotesize,
    label distance=-1mm,
    ]
        \coordinate (Dv1) at (-.5,-.2);
        \coordinate (Dv2) at (-1,1.7);
        \coordinate (Dv3) at (1,2.5);
        \coordinate (Dv4) at (3.5,1);
        \coordinate (Dv5) at (3.2,.-.2);
        \coordinate (lw) at (1.5,-.5);
        \filldraw[fill=black!10!white] 
        (Dv1)--(Dv2)--(Dv3)--(Dv4)--(Dv5)--(lw)--cycle;

        \coordinate (dv1) at (0,0);
        \coordinate (dv2) at (0,1);
        \coordinate (dv3) at (1,1.5);
        \coordinate (dv4) at (2,.5);
        \filldraw[fill=green!20!white]
        (dv1)--(dv2)--(dv3)--(dv4)--(lw)--cycle;

        \draw [thick] (0,-1)--(0,3)
            node[at end, label=right:$\partial\overline H$]{};
        \draw [thick] ($(lw)!-.7!(dv4)$) -- ($(lw)!3!(dv4)$)
            node[at end, label=right:$\partial \overline G$]{};
        \draw [very thick] (dv1)--  (dv2) 
            node[midway,label={[label distance=-2mm]left:$l(\varepsilon)$} ]{};
        \draw [very thick] (lw)-- (dv4)
            node[midway, label={[label distance =-1.5mm]right:$l(\varepsilon')$}]{};

        \node [dot, label=-45:$l(w)$] at (lw) {};
        \node [dot] at (dv4){};
        \node [dot] at (dv1){};
        \node [dot] at (dv2){};
        \node at (1,.5) {$l(\delta)$};
        \node at (1,1.9) {$l(\Delta)$};

        \draw [dashed, ->] (.7,3.05) to[bend left=20] (2.8,2.5);
    \end{tikzpicture}
    \caption{The image of $\Delta$ and $\delta$ under $l$.}
    \label{fig:subdivision proof}
\end{figure}

Moreover, the polygon $l(\delta)$ is contained in the half-space $\overline H\coloneqq \R_{\geq 0}\times \R$ and $l(w)$ a point of $l(\varepsilon')$ of minimal first coordinate. Let $\overline G$ be the half-space in $\R^2$ obtained by shifting $\overline H$ to the right (the direction of increasing first coordinate) until it contains $l(w)$ in its boundary and then rotating it around $l(w)$ until $\partial \overline G$ contains $l(\varepsilon')$ (rotate in the direction that decreases the area of $\overline G\cap l(\Delta)$). By construction, we have  $l(\varepsilon)\nsubseteq\overline G\cap l(\Delta)\subseteq \overline H$. We conclude that $G\coloneqq l^{-1}\overline G$ is a half-space with $\varepsilon'\subseteq \partial G$, we have 
\[
G\cap \Delta\subseteq l^{-1}\overline H \cap \Delta= H\cap \Delta \ ,
\]
and $\varepsilon$ (and in particular one of its vertices) is not contained in $G\cap \Delta$, that is $G\cap V(\Delta)\subsetneq H\cap V(\Delta)$. This contradicts the minimality property of $\varepsilon$ and $H$.

We can now assume that $\varepsilon$ is the unique edge of $\delta$ that is not an edge of $\Delta$. Let $v$ be a vertex of $\delta$ not contained in $\varepsilon$, let $\varepsilon_1$ and $\varepsilon_2$ be the two edges of $\delta$ incident to $v$, and let $F$ be the minimal face of $\Delta$ containing both $\varepsilon_1$ and $\varepsilon_2$. As $\varepsilon_1$ and $\varepsilon_2$ span $\delta$, $F$ contains $\varepsilon$. But this is a contradiction, because this implies that $\varepsilon\in \mathscr{S}|_F$ and $F$ is a distance-$2$-witness by construction.
\end{proof}

\begin{remark}\label{rem:subdivisions}
In the rest of this work, we emphasize \emph{regular subdivisions} of polytopes. That is, subdivisions induced by a height function $\mu\colon V(\Delta)\to \mathbb{R}$, where the domain is the vertex set of our polytope $\Delta$. Often, height functions are interpreted as elements of $\mathbb{R}^{V(\Delta)}$, which equips $\mathbb{R}^{V(\Delta)}$ with a complete fan structure such that the relative interior of each cone corresponds to a combinatorial type of subdivision. This fan structure is known as the \emph{secondary fan} of $\Delta$. Note that maximal cones of the secondary fan correspond to regular triangulations. However, not all subdivisions, or even triangulations, of a given polytope are regular. In fact, the results of Section \ref{sec:one_skeletal} require no assumption of regularity. A standard source for everything related to polyhedral subdivisions is \cite{triangulations}.    
\end{remark}
\section{The Minuscule Types}

In this section, we study the Coxeter Dressians in all minuscule types. We start with a review of types $(A_{n-1},P_k)$ and $(D_n,P_{n-1}) \cong(D_n,P_{n})$, which have been treated in the literature. We then continue with the remaining types. 

Motivated by \Cref{thm:1-skeletal subdivision}, we make the following definition:

\begin{definition}
Let $(\Phi,P)$ be a minuscule type with polytope $\Delta(W(\Phi)/P)$. We define the \emph{distance-$2$ relations} for $(\Phi,P)$ as the set of those strong exchange relations that are supported on the distance-$2$-witnesses of $\Delta(W(\Phi)/P)$.
\end{definition}

\subsection{Type $(A_{n-1},P_k)$}

 As we have seen in \Cref{ex:matroids classical}, the polytope $\Delta(W(A_{n-1})/P_k)$ is the hypersimplex $\Delta_{n,k}$, that is the polytope in $\R^n$ (before, we considered it as a polytope in $\R^n/\R \mathbf 1$, but the quotient map is an isomorphism on the polytope) with vertices $e_I$ for $I\in \binom{[n]}{k}$. The strong exchange relations $\mc F_{A_{n-1},P_k}$ are therefore equations in variables $x_I$ for $I\in \binom{[n]}{k}$, with $x_I$ having weight $e_I$.
They are given by
\begin{equation}
\label{eq:Grassmann-Pluecker}
    f^A_{I,J}\coloneqq
        \bigoplus_{i\in J\setminus I} x_{I\cup i} \odot x_{J\setminus i},
\end{equation}
for all choices of $I\in \binom{[n]}{k-1}$ and $J\in \binom{[n]}{k+1}$.
These are precisely the \emph{(tropical) Grassmann--Plücker relations}. The distance-$2$ relations in type $(A_n,P_k)$ are precisely the $3$-term Grassmann--Plücker relations, which are those $f^A_{I,J}$ with $|J\setminus I|=3$. They are supported on hypersimplices isomorphic to $\Delta_{4,2}$, which are isomorphic to the octahedron $\Diamond_3$.

Over the Boolean semifield $\B=\{0,\infty\}$, the Grassmann--Plücker relations coincide with the symmetric exchange relations for matroids, which are, a priori, stronger that the strong exchange axiom. They are, however, known to be equivalent to the (non-strong) matroid axioms \cite{Brualdi}.

The  relation $f^A_{I,J}$ is homogeneous of weight $e_I+e_J$. Therefore, \Cref{thm:cells_strong_exchange} applies, recovering the result that the cells of the regular subdivision induced by $\mu\in \Dr(A_{n-1},P_k)$ are matroid polytopes. 
Combining \Cref{thm:1-skeletal subdivision} with \Cref{thm:subidivions type DP1} for $\Diamond_3$, proved below, recovers the following theorem, originally due to Speyer \cite{speyer_tropical_linear_spaces}.

\begin{theorem}
\label{thm:subidivions type A}
Let $\mu\in \R^{\binom{[n]}{k}}$. The following are equivalent:
\begin{enumerate}[label=(\alph*)]
    \item $\mu$ satisfies the type $(A_{n-1},P_k)$ distance-$2$ relations, that is the $3$-term  Grassmann--Plücker relations $f^A_{I,J}$ for $I,J$ with $|J\setminus I|=3$,
    \item the subdivision of $\Delta(W(A_{n-1})/P_k)=\Delta_{n,k}$ induced by $\mu$ is strongly matroidal,
    \item the subdivision of $\Delta(W(A_{n-1})/P_k)=\Delta_{n,k}$ induced by $\mu$ is matroidal.
\end{enumerate}
\end{theorem}

Based on work of Murota \cite{Murota}, Rincón proved that the $3$-term Grassmann--Plücker relations are in fact equivalent to all Grassmann--Plücker relations:

\begin{theorem}[{\cite[Corollary 5.2]{rincon_isotropical}}]
\label{thm:3-term Pluecker}
Let $\mu\in\R^{\binom{[n]}{k}}$. Then $\mu\in\Dr(A_{n-1},P_k)$ if and only if $\mu$ satisfies the distance-$2$, that is the $3$-term Grassmann-Plücker relations.
\end{theorem}

In the literature, the Dressian $\Dr(n,k)$ is defined by the Grassmann--Plücker relations, hence $\Dr(n,k)=\Dr(A_n,P_k)$. If one allows for infinite coordinates, the $3$-term Pl\"ucker relations do not suffice for cutting out the Dressian.

\subsection{Types $(D_n,P_{n-1})$ and $(D_n,P_n)$}
We have seen in \Cref{ex:matroids classical}, that the polytope  $\Delta(W(D_n)/P_{n-1})$ (resp.\ $\Delta(W(D_n)/P_{n})$) is the demicube in the $n$-cube $\left[-\frac 12, \frac 12\right]^n$ spanned by the all vertices of the cube with an odd (resp.\ even) number of negative coordinates. We identify $W(D_n)/P_{n-1}$ (resp.\ $W(D_n)/P_{n}$) with the set $2^{[n]}_\odd$ (resp.\ $2^{[n]}_\even$)  of all subsets of $[n]$ of odd (resp.\ even) cardinality by recording which coordinates of the corresponding vertex of the cube are negative. As the two types $(D_n,P_{n-1})$ and $(D_n,P_n)$ are completely analogous, we focus on the case latter. The strong exchange equations $\mc F_{D_n,P_n}$ are in the variables $x_I$ for $I\in 2^{[n]}_\even$. 
They are given by
\begin{align}\label{eq:D_n_strong}
    f_{I,J}^{D} \coloneqq 
        \bigoplus_{i \in I\triangle J} 
            x_{I \triangle i} \odot x_{J \triangle i}  
\end{align}
for $I,J \subseteq [n]$ with $|I \triangle J |\geq 4$ and $I,J$ of odd cardinality.  Here, $I\triangle J=(I\cup J)\setminus (I\cap J)$ denotes the symmetric difference. The relations $f_{I,J}^D$ are precisely the \emph{(tropical) Wick relations} from \cite{rincon_isotropical}.
Over the Boolean semifield $\B=\{0,\infty\}$, they are, a priori, stronger than the strong exchange axioms for Delta matroids, but by \cite[Theorem 4.2.4]{BorovikGelfandWhite:2003} they are equivalent to the (non-strong) Delta matroid axioms.

The weight of the variable $x_I$ is $v_I\coloneqq \frac 12(e_I-e_{[n]\setminus I})$. With these weights, the relation $f_{I,J}^D$ is homogeneous of weight $v_I+v_J$. Therefore, \Cref{thm:cells_strong_exchange} applies and the subdivision of $\Delta(W(D_n)/P_n)$ induced by $\mu\in \Dr(D_n,P_n)$ is strongly matroidal. The distance-$2$ relations are precisely the $4$-term Wick relations, that is those $f^D_{I,J}$ with $|I\triangle I|=4$. These are supported on faces that are $4$-demicubes, and hence isomorphic to $\Diamond_4$.

\begin{theorem}[{\cite[Section 5]{rincon_isotropical}}]
\label{thm:subidivions type Dn}
Let $\mu\in \R^{2^{[n]}_\even}$ (the analogous result holds for $\mu\in \R^{2^{[n]}_\odd}$). Then the following are equivalent:
\begin{enumerate}[label=(\alph*)]
    \item $\mu$ satisfies type $(D_n,P_n)$ distance-$2$ relations, that is the $4$-term Wick relations $f^D_{I,J}$ for $I,J$ with $|I\triangle J|=4$,
    \item the subdivision of $\Poly{D_n}{P_n}$ induced by $\mu$ is strongly matroidal,
    \item the subdivision of $\Poly{D_n}{P_n}$ induced by $\mu$ is matroidal.
\end{enumerate}
\end{theorem}

Similar as for type $(A_n,P_k)$, Rincón showed, based on Murota's work \cite{Murota}, that the $4$-term Wick relations imply all remaining Wick relations.

\begin{theorem}[{\cite[Section 5]{rincon_isotropical}}]
\label{thm:distance-2 Dn}
Let $\mu\in \R^{2^{[n]}_\even}$. Then $\mu\in \Dr(D_n,P_n)$ if and only $\mu$ satisfies all distance-$2$ relations, that is all $4$-term Wick relations.
\end{theorem}

\subsection{Type $(B_n,P_n)$}
Recall from \Cref{ex:weights classical} and \Cref{ex:matroids classical} that in type $B_n$, the Weyl group is the signed symmetric group and the only minuscule weight is $\omega_n$, with $\Poly{B_n}{P_n}$ the $n$-cube $\left[-\frac 12,\frac 12\right]^n$. As in the examples, we label the vertices of $\Poly{B_n}{P_n}$ by elements in $2^{[n]}$ by recording which coordinates are positive. Therefore, the strong exchange equations in $\mc F_{B_n,P_n}$ are in variables $x_I$ for $I\in 2^{[n]}$, with $x_I$ having weight $v_I=\frac 12(e_I-e_{[n]\setminus I})$.
The equations are as follows:
\begin{align}\label{eq:B_n_strong}
    f_{I,J}^{B} \coloneqq 
    \begin{cases}
        \bigoplus_{i \in I\triangle J} 
            x_{I \triangle i} \odot x_{J \triangle i}  &\text{ if } |I \triangle J | \equiv 0 \mod 2 \\
        x_I \odot x_J \oplus \bigoplus_{i \in I\triangle J} 
            x_{I \triangle i} \odot x_{J \triangle i} &\text{ if } |I \triangle J | \equiv 1 \mod 2 ,
    \end{cases}
\end{align}
for all $I,J\subseteq [n]$ with $|I \triangle J |\geq 3$.

The strong exchange relations in \eqref{eq:B_n_strong} are homogeneous of weight $v_I+v_J$. In particular, they are affinely invariant by \Cref{lem:homogeneous implies affinely invariant}. So by \Cref{thm:cells_strong_exchange}, we obtain the following:

\begin{proposition}
\label{prop:subdivisions for Bn}
Every $\mu\in \Dr(B_n,P_n)$ induces a strongly matroidal subdivision of the $n$-cube.
\end{proposition}

It is now natural to ask whether type $B_n$ behaves as nicely as type $A_n$, and if the converse of \Cref{prop:subdivisions for Bn} also holds. This turns out to be wrong already for $n=3$, as Example \ref{ex:strong_fails_fink} shows.

\begin{example}
\label{ex:strong_fails_fink}
Type $B_3$ yields a single strong exchange equation which is supported on the entire $3$-cube. The equation is given by
\[
    f^{B}_{\emptyset,[3]} = 
        x_{\emptyset}\odot x_{123}\oplus 
        x_{1}\odot x_{23}\oplus 
        x_{2}\odot x_{13}\oplus 
        x_{3}\odot x_{12}.
\]
The vector $\mu$ below
\[
    (\mu_{\emptyset},\mu_1,\mu_2,\mu_3,\mu_{12},\mu_{13},\mu_{23},\mu_{123})=(0,1,2,3,0,0,0,4)\, .
\]
does not satisfy the strong equation above, that is $\mu\in \Dr(B_3,P_3)$.
On the other hand, there are $5$ maximal cells in the induced subdivision, each of which a tetrahedron. Notice that none of these cells contain any antipodes, and therefore each of them satisfies the strong exchange property vacuously. We visualize the subdivision in Figure \ref{fig:3_cube_subdiv}. However, we observe that there \emph{is} an element of $\Dr(B_3,P_3)$ that induces this subdivision. For example, $\eta$ given by
\[
    (\eta_{\emptyset},\eta_1,\eta_2,\eta_3,\eta_{12},\eta_{13},\eta_{23},\eta_{123})=(0,1,1,1,0,0,0,1)\, .
\]
In fact, in Section \ref{sec:B3_comps}, we will see that the subdivision from Figure \ref{fig:3_cube_subdiv} has a representative height function which is a ray of $\text{Dr}(B_3,P_3)$. However, as this is a triangulation, the secondary fan structure for the $3$-cube would associate such a subdivision with the interior of a full dimensional cone.
\end{example}

\begin{figure}
    \centering
    \includesvg[width=0.3\textwidth]{triangulation_cube.svg}
    \caption{A strong exchange triangulation of the $3$-cube.}\label{fig:3_cube_subdiv}
\end{figure}
\begin{example}\label{ex:BGW_counter_example}
Let $\mathcal{M}$ be the type $B$ Coxeter matroid whose bases correspond to the vertices of $\Delta(W(B_3)/P_3)$ indexed by $\emptyset,\{1\},\{2\},\{3\}$, and $\{1,2,3\}$. Notice that $f^{B}_{\emptyset,[3]}$ has unique minimum when evaluated at $\mu_{\mathcal{M}}$. Hence, it is not a strong Coxeter matroid. This is a counter example to \cite[Section 6.16, Open Problem 1]{BorovikGelfandWhite:2003}, where it is conjectured that a Coxeter matroid $\mathcal{M}\subset W/P$ has the strong exchange property if and only if $P$ is minuscule. This very same Coxeter matroid is depicted in Figure $4.1$ of loc.\ cit. 
\end{example}
\begin{example}\label{ex:B4_equations}

There are $10$ equations in $\mc F_{B_4,P_4}$. Of these, $8$ equations are supported on a $3$-dimensional face of the $4$-hypercube:
\begin{align*}
f_{1,1234} &= x_1 \odot x_{1234} \oplus x_{12} \odot x_{134} \oplus x_{13} \odot x_{124} \oplus x_{14} \odot x_{123}\\
f_{\emptyset,234} &= x_{\emptyset} \odot x_{234} \oplus x_{2} \odot x_{34} \oplus x_{3} \odot x_{24} \oplus x_{4} \odot x_{23}\\
f_{2,1234} &= x_2 \odot x_{1234} \oplus x_{12} \odot x_{234} \oplus x_{23} \odot x_{124} \oplus x_{24} \odot x_{123}\\
f_{\emptyset,134} &= x_{\emptyset} \odot x_{134} \oplus x_{1} \odot x_{34} \oplus x_{3} \odot x_{14} \oplus x_{4} \odot x_{13}\\
f_{3,1234} &= x_3 \odot x_{1234} \oplus x_{13} \odot x_{234} \oplus x_{23} \odot x_{134} \oplus x_{34} \odot x_{123}\\
f_{\emptyset,124} &= x_{\emptyset} \odot x_{124} \oplus x_{1} \odot x_{24} \oplus x_{2} \odot x_{14} \oplus x_{4} \odot x_{12}\\
f_{4,1234} &= x_4 \odot x_{1234} \oplus x_{14} \odot x_{234} \oplus x_{24} \odot x_{134} \oplus x_{34} \odot x_{124}\\
f_{\emptyset,123} &= x_{\emptyset} \odot x_{123} \oplus x_{1} \odot x_{23} \oplus x_{2} \odot x_{13} \oplus x_{3} \odot x_{12}\, .
\end{align*}
Moreover, there are two equations whose support is not contained in a proper face:
\begin{align*}
f_{\emptyset,1234}&= x_1 \odot x_{234} \oplus x_{2} \odot x_{134} \oplus x_{3} \odot x_{124} \oplus x_{4} \odot x_{123}\\
f_{1,234}&=x_{\emptyset} \odot x_{1234} \oplus x_{12} \odot x_{34} \oplus x_{13} \odot x_{24} \oplus x_{14} \odot x_{23}\, .
\end{align*}
\end{example}

\begin{lemma}\label{lem:four_cube_equations}
Let $\mu \in \R^{2^{[4]}}$ satisfy the all strong exchange equations supported on the $3$-dimensional faces of the $4$-cube. Then $\mu$ is already a  valuated Coxeter matroid.
\end{lemma}

\begin{proof}
Let $\mu\colon \R^{2^{[4]}} \to  \R$. By symmetry of the cube, it suffices to show that $\mu$  satisfies $f_{1,234}$. 
By symmetry, we assume that the minimum in the equation corresponding to $f_{1,234}$ is attained only at $\mu_{\emptyset} + \mu_{1234}$.
Because the strong exchange relations in $\mc F_{B_4,P_4}$ are homogeneous, they are affinely invariant by \Cref{lem:homogeneous implies affinely invariant}.
Hence, we can replace $v$ by adding an affine linear function such that $\mu_{\emptyset} = \mu_{12} = \mu_{13} = \mu_{14} = \mu_{1234}=0$, since the space of affine linear functions on $\R^4$ is $5$-dimensional. Thus, as the minimum in question is only attained once, it follows $\mu_{23},\mu_{24},\mu_{34}>0$.

Let 
$m = \min(\mu_{1},\mu_{2},\mu_{3},\mu_{4},\mu_{123},\mu_{124},\mu_{134},\mu_{123})$. 

\begin{itemize}
\item By equation $f_{\emptyset,234}$ it then follows, that one of $\mu_2,\mu_3,\mu_4$ is strictly smaller than $\mu_{234}$. Without loss of generality, we assume $\mu_2 < \mu_{234}$, in particular $m< \mu_{234}$.
\item By equation $f_{2,1234}$, we deduce that $m<\mu_2$. Analogously, by equations $f_{3,1234},f_{4,1234}$, that $m < \mu_3,\mu_4$.
\item By equations $f_{\emptyset,134},f_{\emptyset,124},f_{\emptyset,123}$ respectively we obtain that $m<\mu_{134},\mu_{124},\mu_{123}$.
\item By exclusion, $\mu_1$ is the unique minimum, and in particular $\mu_1<\mu_{134},\mu_{124},\mu_{123}$ from the previous step. This however, contradicts equation $f_{1,1234}$.
\end{itemize}
\end{proof}

In \cite{Fink_et_al}, the authors attribute an idea inspiring the following Proposition to Matt Larson.

\begin{proposition}
\label{prop:B_n_vs_D_n+1}
The isomorphism $\T^{2^{[n]}}\hookrightarrow \T^{2^{[n+1]_\even}}$ induced by the bijection
\[
    2^{[n]}\longrightarrow 2^{[n+1]}_\even, \quad A \mapsto \overline{A} = 
    \begin{cases}
        A  & 
            \text{ if } |A| \text{ even} \\
        A\cup \{n+1\} & 
            \text{ if } |A| \text{ odd.} \\
    \end{cases}
\]
induces a linear isomorphism $\Dr(B_n,P_n) \cong \Dr(D_{n+1},P_{n+1})$.
\end{proposition}
\begin{proof}
We first show that type $B_n$ strong exchange equations are in bijection to type $D_{n+1}$ strong exchange equations via $f_{I,J}^{B}\mapsto f_{I',J'}^D$. Here $I'= I \cup \{n+1\}$ if $|I|$ even, and $I'=I$ if $|I|$ odd, and define $J'$ similarly. Bijectivity follows from the fact that $I'$, $J'$ are always of odd cardinality, and $|I\triangle J|\geq 3$ if and only if $|I' \triangle J'|\geq 4$. 

Let $\mu \in \Dr(B_n,P_n)$, and $\overline{\mu} \in \T^{2^{[n+1]}}$ its corresponding image. We now show that $\mu$ satisfies $f_{I,J}^{B}$ if and only if $\overline{\mu}$ satisfies $f_{I',J'}^D$. If both $I,J$ are of odd cardinality, then the respective equations are the same, and there is nothing to show. If $I$ is odd, and $J$ is even, then $I'=I$, and $J'=J\cup \{n+1\}$. In that case, $I'\triangle J' = I\triangle J \cup \{n+1\}$ is of odd cardinality. Hence, for all $i \in I \triangle J$, we have $I \triangle\{i\}$ is even, and $\overline{J \triangle i} = J' \triangle i $. Finally, for $i=\{n+1\}$, we have $I \triangle i = I$, and $J = J'\triangle i$, finishing the proof for the case $|I|$ even, $|J|$ odd. Lastly, let $I,J$ both be even. Then $I \triangle J= I'\triangle J'$, and for all $i\in I \triangle J$ we have $\overline{I\triangle i}=I'\triangle i$, finishing the proof.
\end{proof}

We note that the above is already true algebraically, as the same morphism over any field maps the minuscule type $B_n$ flag variety isomorphically onto the even component of the spinor variety. There is also an equivalent map $\Dr(B_n,P_n)\cong \Dr(D_{n+1},P_{n})$, mapping subsets of $[n]$ to odd subsets of $[n+1]$. As the even and odd components of the spinor variety are isomorphic, this is not surprising. The following can be deduced immediately from Proposition \ref{prop:B_n_vs_D_n+1} and \cite[Theorem 4.6]{rincon_isotropical}.

\begin{corollary}
\label{cor:Realizability type Bn}
For $n\leq 4$, the Coxeter Dressian $\Dr(B_n,P_n)$ is the tropicalization of the minuscule type $(B_n,P_n)$ flag variety. For $n\geq 6$, it is strictly larger than the tropicalization. In particular, not every valuated Coxeter matroid of type $B_n$ is realizable for $n\geq 6$.
\end{corollary}

In type $B_n$, we can pass to a smaller set of equations cutting out the same solution set. The following can be seen as a tropical version of \cite[Corollary 4.3]{universal_equations_isotropic}.

\begin{theorem}\label{thm:b_4_equations}
Let $\mu \in \T^{2^{[n]}}$, such that $\supp(\mu)$ is a strong Coxeter matroid of type $B_n$ . Then the following statements are equivalent:
\begin{enumerate}[label=(\alph*)]
    \item $\mu$ is in the Coxeter Dressian,
    \item $\mu$ satisfies the \emph{four term strong exchange equations} $f_{I,J}^{B}$ for $|I \triangle J| \leq 4$.
\end{enumerate}
Moreover, if all entries of $\mu$ are finite, in (b) it suffices to only consider the case $|I \triangle J| \leq 3$.
\end{theorem}
\begin{proof}
The statement (a) $\Rightarrow$ (b) holds by definition. For the other direction, we need to show that $\mu$ satisfies all strong exchange equations $f^B_{I,J}$. Consider the image $\overline{\mu}$ from Proposition \ref{prop:B_n_vs_D_n+1}. By \cite[Theorem 5.1]{rincon_isotropical}, $\overline{\mu}$ satisfies the type $D_{n+1}$ strong exchange equations if and only if it satisfies the respective four-term strong exchange equations. Hence $\overline{\mu}\in \Dr(D_{n+1},P_{n+1})$, and we conclude $\mu \in \Dr(B_{n},P_{n})$.

Let additionally all entries of $\mu$ be finite. Restricting $\mu$ to the various $4$-cubes and applying Lemma \ref{lem:four_cube_equations} we obtain that $\mu$ satisfies all equations $f_{I,J}^B$ with $|I \triangle J| = 4$ if it does so for all $|I \triangle J |=3$.
\end{proof}

\begin{example}
The point $\mu \in \T^{2^{[4]}}$ given by $\mu_{\emptyset}=\mu_{1234}=\mu_{12}=0$, $\mu_{34}=1$, and $\mu_A=\infty$ for all other $A\in 2^{[n]}$ does not satisfy $f^B_{1,234}$. Hence, it is not in the Coxeter Dressian. However, it satisfies all equations supported on proper $3$-faces of the $4$-cube. 
\end{example}

In the subsequence sections, we discuss the results of direct computation via \texttt{OSCAR}.

\subsection{Direct computations for $B_3$}\label{sec:B3_comps}

We start with the case where $n=3$ and the ambient polytope is the $3$-cube, $\Delta(W(B_3)/P_3)=[-\frac{1}{2},\frac{1}{2}]^3$ . We begin by
computing the tropical prevariety $\text{Dr}(B_3,P_3)$. Since $\Dr(B_3,P_3)$ is a tropical hypersurface, it has a unique coarsest polyhdral structure inherited from the Gr\"obner fan of the associated principal ideal. However, we observe that the lineality space contains height functions inducing non-trivial strong exchange subdivisions of $\Delta(W(B_3)/P_3)$. Hence, a more refined structure is desireable for our purposes.

We compute the entirety of the subfan of the secondary fan of $\Delta(W(B_3)/P_3)$ whose cones correspond to strong matroidal subdivisions. We call this the \emph{strong exchange fan} of $\Delta(W(B_3)/P_3)$ and denote it by $\strongExchangeFan$. To endow the support of $\Dr(B_3,P_3)$ with a structure which captures induced subdivisions, we take the common refinement of $\strongExchangeFan$ with $\Dr(B_3,P_3)$. We denote this fan by $\refinedDressian$. 

\begin{proposition}\label{prop:B3_computations}
We use \texttt{OSCAR} to obtain the following.
    \begin{enumerate}[label=(\alph*)]
        \item The polyhedral complexes $\Dr(B_3,P_3)$, $\strongExchangeFan$, and $\refinedDressian$ have $f$-vectors
        \begin{align*}
        f_{\Dr(B_3,P_3)}&= (0,0,0,0,0,1,4,6)\\
        f_{\strongExchangeFan} &= (0,0,0,0,1,14,40,32,2)\\
        f_{\refinedDressian}&= (0,0,0,0,1,16,48,36),
        \end{align*}
        so $\strongExchangeFan$ is $8$-dimensional counting lineality, and $\refinedDressian$ is $7$-dimensional.
        \item The strong exchange fan $\strongExchangeFan$ is not pure, but $\refinedDressian$ is pure.
        \item The interior of every cone of $\strongExchangeFan$ has a representative height function which is in the support of $\refinedDressian$. 
    \end{enumerate}
\end{proposition}
By Proposition \ref{prop:B3_computations}, we have a set theoretic containment of $\refinedDressian$ into the strong exchange fan.  Example \ref{ex:strong_fails_fink} shows us that the support of $\refinedDressian$ is not equal to the support of $\strongExchangeFan$. However, in the $n=3$ case, Proposition \ref{prop:B3_computations} (c) tells us that $\refinedDressian$ intersects every cone of the strong exchange fan non-trivially. Furthermore, points in the relative interiors of each maximal cone of $\refinedDressian$ induce unique subdivisions of the $3$-cube. So each maximal polyhedron is contained in a unique secondary cone.

There is an obvious action of the signed permutation group $W(B_3)$ on the $3$-dimensional cube $\Delta(W(B_3)/P_3)$. Hence, this group naturally acts on both the Coxeter Dressian $\Dr(B_3,P_3)$ (as a set) and the strong exchange fan $\Sigma^S_{B_3}$. Note that $\refinedDressian$ is the coarsest polyhedral structure on $\Dr(B_3,P_3)$ such that every cone is a subcone of a secondary cone. Hence, we have an induced action of $W(B_3)$ on $\refinedDressian$.

\begin{proposition}\label{prop:orbits}
Under the $W(B_3)$-action, there are 3 orbits of rays of $\refinedDressian$, and 3 orbits of maximal cones.
\end{proposition}
The $3$ orbits of rays yield different subdivisions of the $3$-cube $\Delta(W(B_3)/P_3)$. Figure \ref{fig:ray_orbits} visualizes these orbits. We remark that $\refinedDressian$ has $2$ rays (the rightmost orbit in Figure \ref{fig:ray_orbits}) which are not rays of the secondary fan. As these rays induce triangulations, they are in fact contained in the relative interior of maximal cones in the secondary fan. In Figure \ref{fig:max_orbits}, we see the same subdivision as an orbit of a maximal cell of $\refinedDressian$.
Notice that this subdivision is combinatorially isomorphic to the one induced by by the height function in Example \ref{ex:strong_fails_fink}. 
The other two rays induce splits of the cube, in the sense of \cite{splittingpolytopes}. As these are among the coarsest nontrivial subdivisions, they are also rays of the secondary fan. 

\begin{figure}[htbp]
  \centering
  \begin{subfigure}[t]{0.32\textwidth}
    \centering
    \includesvg[width=\textwidth]{ray_orbit_8.svg}
    \caption{Orbit length $8$.}
  \end{subfigure}\hfill
  \begin{subfigure}[t]{0.32\textwidth}
    \centering
    \includesvg[width=0.8\textwidth]{ray_orbit_6.svg}
    \caption{Orbit length $6$.}
  \end{subfigure}\hfill
  \begin{subfigure}[t]{0.32\textwidth}
    \centering
    \includesvg[width=0.8\textwidth]{triangulation_cube.svg}
    \caption{Orbit length $2$.}
  \end{subfigure}
  \caption{The three strong matroidal subdivisions of the $3$-cube obtained from rays in $\refinedDressian$.}
  \label{fig:ray_orbits}
\end{figure}

Similarly, we now consider maximal cones of $\refinedDressian$. By Proposition \ref{prop:orbits}, there are $3$ subdivisions of the $3$-cube induced by maximal cones of $\refinedDressian$, up to symmetry. These are visualized in Figure \ref{fig:max_orbits}.

\begin{figure}[htbp]
  \centering
  \begin{subfigure}[t]{0.32\textwidth}
    \centering
    \includesvg[width=\textwidth]{max_polys_orbit1.svg}
    \caption{Orbit length $12$.}
  \end{subfigure}\hfill
  \begin{subfigure}[t]{0.32\textwidth}
    \centering
    \includesvg[width=0.9\textwidth]{max_polys_orbit2.svg}
    \caption{Orbit length $12$.}
  \end{subfigure}\hfill
  \begin{subfigure}[t]{0.32\textwidth}
    \centering
    \includesvg[width=0.8\textwidth]{triangulation_cube.svg}
    \caption{Orbit length $12$.}
  \end{subfigure}
  \caption{The three strong matroidal subdivisions of the $3$-cube obtained from maximal cells of $\refinedDressian$.}
  \label{fig:max_orbits}
\end{figure}
We observe that the first two subdivisions are not triangulations, hence there are maximal cones of $\refinedDressian$ that are subcones of non-maximal cones of $\strongExchangeFan$.

\subsection{Direct computations for $B_4$}
For $n=4$, we compute two tropical prevarieties whose support is $\Dr(B_4,P_4)$. 
We summarize these results in the following proposition.

\begin{proposition}
\ 
\begin{enumerate}[label=(\alph*)]
\item The tropical prevariety associated to the tropical polynomials in $\mc F_{B_4,P_4}$ has $f$-vector
\[(0,0,0,0,0,0,1,36,280,960,1540,912)\, .\]
\item The tropical prevariety associated to the relations supported on $3$-faces of the $4$-cube, i.e., $\mc F_{B_4,P_4}\setminus \{f_{\emptyset,1234}^B,f_{1,234}^B\}$ has $f$-vector 
\[(0,0,0,0,0,0,1,36,280,960,1528,900)\, .\]
\end{enumerate}
By Lemma \ref{lem:four_cube_equations}, these prevarieties have the same support $\Dr(B_4,P_4)$.
\end{proposition}

\subsection{Type $(C_n,P_1)$}
\label{sec:type_C}
We have already seen in \Cref{ex:weights classical} and \Cref{ex:matroids classical} that the polytope $\Poly{C_n}{P_1}$ is the cross polytope
\[
    \Diamond_n=\conv\{\pm e_{i}:i\in [n]\}.
\]
There is no strong exchange relation for type $(C_n,P_1)$, so the Dressian $\Dr(C_n,P_1)$ is the full secondary fan of $\Diamond_n$. 

On the other hand, every convex hull of vertices of $\Diamond_n$ is the polytope of a strong Coxeter matroid, so every subdivision of $\Diamond_n$ is strongly matroidal. 

\subsection{Type $(D_n,P_1)$}
By \Cref{ex:matroids classical}, the polytope $\Poly{D_n}{P_1}$ is again the $n$-cross polytope $\Delta_n$. We label the vertices $e_i$ by $i\in [n]\cup -[n]$, where we use the convention that $e_{-i}=-e_i$. There is a unique strong exchange equation defining $\Dr(D_n,P_1)$ in variables $x_i$, with $i\in [n]\cup -[n]$. Namely, we have the equation
\begin{align}\label{eq:D_1_strong}
f^D_{\Diamond_n}=\bigoplus^n_{i=1}x_{i}\odot x_{-i}.
\end{align}
The variable $x_{i}$ has weight $e_{i}$ so that the relation $f^D_{\Diamond_n}$ is homogeneous of weight $0$. In particular,\Cref{thm:cells_strong_exchange} applies and the subdivision of $\Diamond_n$ induced by any $\mu\in\Dr(D_n,P_1)$ is strongly matroidal. For type $(D_n,P_1)$ we can show that matroidal and strongly matroidal are in fact equivalent:

\begin{proposition}
\label{prop:matroidal=strongly matroidal for Dn-P1}
Every Coxeter matroid of type $(D_n,P_1)$ is strongly matroidal.
\end{proposition}
\begin{proof}
Let $\mc M\subseteq [n]\cup -[n]$ be a Coxeter matroid of type $(D_n,P_1)$. Assume $\{i,-i\}\subseteq \mc M$ for some $i\in [n]$. We need to show that $\{j,-j\}\subseteq \mc M$ for some $j\in [n]\setminus\{i\}$. So by assumption, there exists $i\in [n]$ such that $e_i$ and $-e_i$ are contained in $\Delta(\mc M)$; this is equivalent to $0\in \Delta(\mc M)$. Let $F$ be a face of $\Delta(\mc M)$, and let $\mc M'$ be the uniquely determined subset of $\mc M$ with $\Delta(\mc M')=F$. All edges of $F$ are also edges of $\Delta(\mc M)$, so $\mc M'$ is a Coxeter matroid. It follows, that after replacing $\mc M$ by $\mc M'$, we may assume that $0$ is contained in the relative interior of $\Delta(\mc M)$ from the onset.

Because $\Delta(\mc M)$ is not equal to $\conv\{e_i,-e_i\}$, there exists $j\in [n]\setminus\{i\}$ with $\pm e_j\in \Delta(\mc M)$. As $0$ is contained in the relative interior, we have $\mp \epsilon e_j\in \Delta(\mc M)$ for some $\epsilon>0$. But this can only happen if $\mp e_j\in \Delta(\mc M)$, that is ${j,-j}\subseteq \mc M$.
\end{proof}

\begin{theorem}
\label{thm:subidivions type DP1}
Let $\mu\in \R^{[n]\cup -[n]}$ be a height function on the vertices of $\Diamond_n$. Then the following are equivalent \begin{enumerate}[label=(\alph*)]
\item $\mu\in \Dr(D_n,P_1)$, that is $\mu$ satisfies $f^D_{\Diamond_n}$
\item The subdivision of $\Diamond_n$ induced by $\mu$ is strongly matroidal,
\item the subdivision of $\Diamond_n$ induced by $\mu$ is matroidal,
\item the subdivision of $\Diamond_n$ induced by $\mu$ is not a triangulation.
\end{enumerate}
\end{theorem}
\begin{proof}
Part (b) follows from (a) by the homogeneity of $f^D_{\Diamond_n}$ and \Cref{thm:cells_strong_exchange}. Parts (b) and (c) are equivalent by \Cref{prop:matroidal=strongly matroidal for Dn-P1}. Now assume (b), that is suppose that the subdivision induced by $\mu$ is strong matroidal. We want to show (a). Let $\mc M\subseteq [n]\cup -[n]$ such that $\Delta(\mc M)$ is a maximal cell of the subdivision. Then there exists $v\in \R^n$ and $\alpha\in \R$ such that
\[
    \langle v, e_i\rangle + \mu_i \geq \alpha \ ,
\]
for all $i\in [n]\cup -[n]$, with equality if and only $i\in \mc M$. We conclude that for $k\in [n]$ we have
\[
\mu_k+\mu_{-k}= \langle v, e_k\rangle +\mu_k + \langle v, e_{-k}\rangle +\mu_{-k}\geq 2\alpha ,
\]
with equality if and only if $\{k,-k\}\subseteq \mc M$.

Now, observe that for $\Delta(\mc M)$ to be full dimensional, it must contain at least $n+1$ vertices of $\Diamond_n$. As a consequence, we must have some some $\{i,-i\}\subseteq \mc M$ for some $i\in [n]$. Furthermore, since $\mc M$ is a strong matroid, there exists $j\in [n]\setminus \{i\}$ such that $\{j,-j\}\subseteq \mc M$. We conclude that the minimum in $f^D_{\Diamond_n}$ is obtained at least twice, namely at least at the summands corresponding to $i$ and $j$.

Finally, we prove the equivalence of $(b)$ and $(d)$. Every maximal cell $C$ of the subdivision induced by $\mu$ is the convex hull of a set of the form
\[
 \{e_i:i\in I\}\cup \{-e_j:j\in J\} \cup \{\pm e_k: k\in K\}
\]
for three pairwise disjoint subsets $I$, $J$, and $K$ of $[n]$ with $K\neq \emptyset$. The cell 
\[
    C_0= \conv\{\pm e_k:k\in K\}
\]
is the unique cell of the subdivision containing $0$ in its relative interior and is therefore independent of the choice of the maximal cell $C$. Moreover, $C$ is a simplex if and only if $\vert K\vert =1$. We conclude that the subdivision is a triangulation, if and only if $\vert K\vert =1$. On the other hand, the maximal cells are strongly matroidal if and only if $\vert K\vert >1$.
\end{proof}

\begin{corollary}\label{corr:Dn_secondaryfan}
The Coxeter Dressian $\Dr(D_n,P_1)$ is supported on the subfan of the secondary fan of $\Diamond_n$ consisting of all cones of codimension at least $1$. Moreover, the coarsest fan structure on $\Dr(D_n,P_1)$ coincides with that of the secondary fan.
\end{corollary}
\begin{proof}
By \Cref{thm:subidivions type DP1}, the Coxeter Dressian consists of all $\mu$ whose induced subdivision is not a triangulation. The interiors of the maximal cones of the secondary fan consist precisely of those heights that induce triangulations, so $\Dr(D_n,P_1)$ is equal to the union of all its codimension-$1$ cells.
The secondary fan is the normal fan of the secondary polytope so that $\Dr(D_n,P_1)$ is a tropical hypersurface supported on the codimension-$1$ skeleton of the normal fan of  the secondary polytope. The assertion about the coarsest fan structure follows from \cite[Proposition 2.4.5]{Hampe2014}.
\end{proof}

\subsection{Computational results for $D_n$ }We directly compute $\Dr(D_{n},P_{1})$ for $n\in\{5,6\}$, and verify Proposition \ref{thm:subidivions type DP1} and Corollary \ref{corr:Dn_secondaryfan}. In these cases, we also study the $f$-vectors of the tropical prevarieties. The results of these computations are summarized below. 
\begin{proposition}\label{prop:Dn_comps}
The $f$-vectors of $\Dr(D_5,P_1)$ and $\Dr(D_6,P_1)$ as subfans of the seconardy fans of $\Delta(W(D_{5})/P_{1})$ and $\Delta(W(D_{6})/P_{1})$  are  
\[(0,0,0,0,0,0,1,5,10,10)\] and \[(0,0,0,0,0,0,0,1,6,15,20,15)\] respectively.
\end{proposition}

\subsection{Types $(E_6,P_1)$ and $(E_6,P_6)$}
\label{sec: E_6}

In type $E_6$, there are two minuscule fundamental weights $\omega_1$ and $\omega_6$. As we have $\Delta(W(E_6)/P_6)=-\Delta(W(E_6)/P_1)$, with a similar duality occurring for the algebraic flag variety, we may treat both cases simultaneously. The polytope $\Delta(W(E_6)/P_1)$ is commonly referred to as the $2_{21}$ polytope, or as the \emph{Schläfli} polytope  \cite{coxeter_2_21}. The face structure of $2_{21}$ is studied in \cite{Cox_regular}. We recall the relevant facts. The polytope $2_{21}$ is $6$ dimensional and has $27$ vertices, each at most edge-distance 2 from each other. Each distance-$2$ pair of vertices defines a $\Diamond_5$ cross polytope in the $5$-skeleton of $2_{21}$, of which there are $27$. All other facets are simplices. For each $\Diamond_5$ facet $F$, we get a strong exchange equation
\begin{align}\label{eq:E_6_strong}
f^{E_{6}}_{F} = \bigoplus_{\substack{\{A,B\}\subseteq F \\ d(A,B)=2}} x_A \odot x_{B}.
\end{align}
Moreover, all strong exchange equations in $\mc F_{E_6,P_1}$ arise in this way. In particular, all strong exchange equations are distance-$2$ relations in this type. 

\begin{theorem}
\label{thm:subidivions type E6}
Let $\mu\in \R^{W(E_6)/P_1}$ be a height function on the vertices of $\Delta(W(E_6)/P_1)$. Then the following are equivalent:
\begin{enumerate}[label=(\alph*)]
    \item $\mu\in \Dr(D_n,P_1)$,
    \item the subdivision of $\Delta(W(E_6)/P_1)$ induced by $\mu$ is strongly matroidal,
    \item the subdivision of $\Delta(W(E_6)/P_1)$ induced by $\mu$ is matroidal.
\end{enumerate}
\end{theorem}
\begin{proof}
We note that the line segment connecting two vertices of $\Delta(W(E_6)/P_1)$ is parallel to a root if and only if it is an edge of the polytope.

If $\mu\in \Dr(D_n,P_1)$ then its induced subdivision is strongly matroidal by the homogeneity of the strong exchange relations. So (a) implies (b). Being strongly matroidal certainly implies being matroidal, so (b) implies (c). Now assume the subdivision of $\mu$ is matroidal. Then so is its restriction to all the faces that are $\Diamond_5$ cross polytopes. By \Cref{thm:subidivions type DP1}, the height function $\mu$ satisfies $f^D_{\Diamond_5}$ supported on these cross polytopes. But these equations coincide with the strong exchange relations in \eqref{eq:E_6_strong}, proving that (c) implies (a).
\end{proof}

\subsection{Type $(E_7,P_7)$}
\label{sec:E_7}
The polytope is commonly known as $\Delta(W(E_7)/P_7)$ is commonly known as the $3_{21}$ polytope or the \emph{Gosset} polytope \cite{Cox_regular}. The polytope is $7$ dimensional, with $56$ vertices with maximum edge distance $3$ between pairs. Each vertex has $27$ neighbors. Each distance 2 pair induces a $\Diamond_6$ facet, and distance $3$ pairs are of the form $(v,-v)$. Since each element $A$ has a unique distance $3$ counterpart $B$, we just write $B=-A$. The strong exchange equations come in two forms. We have an equation for each $\Diamond_6$ facet and for the whole polytope
\begin{equation}\label{eq:E_7_strong}
\begin{aligned}
f^{E_7}_F
  &= \bigoplus_{\substack{\{A,B\}\subseteq F \\ d(A,B)=2}} x_A \odot x_{B}
     &&F \text{ a } \Diamond_6 \text{ facet}, \\[6pt]
f^{E_7}
  &= \bigoplus_{\{A,-A\}\subseteq W/P} x_A \odot x_{-A}.
\end{aligned}
\end{equation}

Over the Boolean semifield, the strong exchange equations \eqref{eq:E_7_strong} are equivalent to the strong exchange axiom of Borovik, Gelfand, and White \cite{SymmetricExchangeAxiom}, which they conjectured to be equivalent to the Coxeter matroid axioms (e.g.\ the maximality property) \cite[Section 6.16]{BorovikGelfandWhite:2003}. By means of an exhaustive search, we prove that this conjecture holds in type $(E_7,P_7)$ in upcoming work by the authors \cite{MinusculeBasisGraphs}.

The distance-$2$ witnesses of $\Delta(W(E_7)/P_7)$ are precisely the $\Diamond_6$ faces on which the relations $f^{E_7}_{A,B}$ are supported. In fact, these relations are precisely the relation $f^{D}_{\Diamond_6}$ in the coordinates corresponding to the vertices of the cross polytope. This allows us to prove the analogue of \Cref{thm:subidivions type A}.

\begin{theorem}
\label{thm:subidivions type E7}
Let $\mu\in \R^{W(E_7)/P_7}$. Then the following are equivalent:
\begin{enumerate}[label=(\alph*)]
    \item $\mu$ satisfies all distance-$2$ relations, that is all relations $f_{A,B}^{E_7}$
    \item the subdivision of $\Delta(W(E_7)/P_7)$ induced by $\mu$ is strongly matroidal,
    \item the subdivision of $\Delta(W(E_7)/P_7)$ induced by $\mu$ is matroidal.
\end{enumerate}
\end{theorem}
\begin{proof}
By upcoming work of the authors \cite{MinusculeBasisGraphs}, being matroidal and strongly matroidal is equivalent in type $(E_7,P_1)$. It is left to show the equivalence of (a) and (c). 

We observe that a line segment between two vertices of $\Delta(W(E_7)/P_7)$ is parallel to a root if and only if that line segment is already an edge. Consequently, a subdivision of $\Delta(W(E_7)/P_7)$ is matroidal if and only if it is $1$-skeletal.
Applying Theorem \ref{thm:1-skeletal subdivision}, we see that $\mu$ induces a matroidal subdivision if and only its induced subdivision on all distance-$2$ witnesses is matroidal. The distance-$2$ witnesses are isomorphic to cross polytopes $\Diamond_6$, so by \Cref{thm:subidivions type DP1}, $\mu$ induces a matroidal subdivision, if and only if it satisfies the equation $f^D_{\Diamond_6}$ for all distance-$2$ vertices. But these equations coincide with the distance-$2$ relations.
\end{proof}

We conjecture that the analogue of \Cref{thm:3-term Pluecker} holds in type $(E_7,P_7)$ as well.

\begin{conjecture}\label{conj:E7}
Let $\mu\in \R^{W(E_7)/P_1}$. The following are equivalent:
\begin{enumerate}[label=(\alph*)]
    \item $\mu$ satisfies all distance $2$-relations $f^{E_7}_{F}$,
    \item $\mu\in \Dr(E_7,P_1)$, that is $\mu$ satisfies all strong exchange relations in $\mc F_{E_7,P_7}$.
\end{enumerate}
\end{conjecture}

\bibliographystyle{amsalpha}
\bibliography{biblio}{}

\end{document}